\PassOptionsToPackage{dvipdfm}{graphicx}
\documentclass[11pt]{amsart}
\usepackage{listings}
\usepackage{amsmath, amssymb, amscd, amsthm, mathrsfs, url ,pinlabel, verbatim, lipsum, wrapfig}
\usepackage[text={6.5in,9.5in}, centering, letterpaper, dvips]{geometry}

\usepackage{bmpsize}

\usepackage{color, multirow, latexsym, float, xcolor}
\usepackage{caption, tikz, tikz-cd, setspace, enumitem}

\usetikzlibrary{positioning, arrows.meta, knots, braids}

\usepackage{scalerel,stackengine}
\stackMath
\newcommand\reallywidehat[1]{%
\savestack{\tmpbox}{\stretchto{%
  \scaleto{%
    \scalerel*[\widthof{\ensuremath{#1}}]{\kern-.6pt\bigwedge\kern-.6pt}%
    {\rule[-\textheight/2]{1ex}{\textheight}}
  }{\textheight}%
}{0.5ex}}%
\stackon[1pt]{#1}{\tmpbox}%
}

\usepackage[backref=page, linktocpage=true]{hyperref}

\renewcommand*{\backref}[1]{}
\renewcommand*{\backrefalt}[4]{%
    \ifcase #1 (Not cited.)%
    \or        (Cited on page~#2.)%
    \else      (Cited on pages~#2.)%
    \fi}

\definecolor{maroon}{rgb}{0.5, 0.0, 0.0}
\definecolor{darkblue}{rgb}{0.03, 0.27, 0.49}

\hypersetup{
  colorlinks   = true, 
  urlcolor     = maroon, 
  linkcolor    = darkblue, 
  citecolor   = maroon 
}

\newcommand{\CommaPunct}{\mathpunct{\raisebox{0.5ex}{,}}}

\usepackage[T1]{fontenc}
\usepackage{lmodern}
\DeclareUrlCommand{\bfurl}{}
\DeclareMathOperator{\Fix}{Fix}

\DeclareMathOperator{\Ker}{Ker}

\DeclareMathOperator{\id}{id}

\newtheorem{thm}{Theorem}[section]

\newtheorem{lem}[thm]{Lemma}
\newtheorem{cor}[thm]{Corollary}
\newtheorem{prop}[thm]{Proposition}

\theoremstyle{definition}
\newtheorem{defn}[thm]{Definition}

\newtheorem{exmp}{Example}[section]

\newtheorem{introthm}{Theorem}

\newtheorem{introprob}{Problem}

\theoremstyle{remark}
\newtheorem{rem}{Remark}
\newcommand{\T}{\mathcal{T}}
\renewcommand{\L}{\mathcal{L}}
\newcommand\Z{\mathbb{Z}}
\newcommand\Q{\mathbb{Q}}
\newcommand\R{\mathbb{R}}
\newcommand\C{\mathcal{C}}
\newcommand\CQ{\mathcal{C}_{\mathbb{Q}}}
\newcommand\CT{\widetilde{\mathcal{C}}}
\newcommand\CTQ{\widetilde{\mathcal{C}}_{\mathbb{Q}}}

\newcommand\J{\mathcal{J}}
\newcommand\JO{\mathcal{J}_{\mathrm{odd}}}
\newcommand\JE{\mathcal{J}_{\mathrm{even}}}
\usepackage[colorinlistoftodos,textwidth=3cm]{todonotes}

\newlength\Colsep
\title{Equivariant $\mathbb{Q}$-sliceness of strongly invertible knots}

\author{Alessio Di Prisa}
\address{Scuola Normale Superiore, 56126 Pisa, Italy}
\email{\url{alessio.diprisa@sns.it}}
\urladdr{\url{https://sites.google.com/view/alessiodiprisa}}

\author{O{\u{g}}uz \c{S}avk}
\address{CNRS and Laboratorie de Math\'ematiques Jean Leray, Nantes Universit\'e, 44322 Nantes, France}
\email{\url{oguz.savk@cnrs.fr}}
\urladdr{\url{https://sites.google.com/view/oguzsavk}}

\date{}

\begin{document}

\begin{abstract}

We introduce and study the notion of equivariant $\mathbb{Q}$-sliceness for strongly invertible knots. On the constructive side, we prove that every \emph{Klein amphichiral knot}, which is a strongly invertible knot admitting a \emph{compatible} negative amphichiral involution, is equivariant $\mathbb{Q}$-slice in a single $\Q$-homology $4$-ball, by refining Kawauchi's construction and generalizing Levine's uniqueness result. On the obstructive side, we show that the equivariant version of the classical Fox-Milnor condition, proved recently by the first author in \cite{DP23b}, also obstructs equivariant $\Q$-sliceness. We then introduce the equivariant $\mathbb{Q}$-concordance group and study the natural maps between concordance groups as an application. We also list some open problems for future study.
\end{abstract}

\maketitle

\section{Introduction}
\label{sec:intro}

A knot $K \subset S^3$ is called \emph{invertible} if $K$ is isotopic to its reverse $-K$. Similarly, $K$ is said to be \emph{negative amphichiral} if $K$ is isotopic to the reverse of its mirror image $-\overline{K}$. When the isotopy maps $\rho$ and $\tau$ are further chosen to be involutions (see {\sc\S}\ref{sec:symmetry}), the pairs $(K, \rho)$ and $(K, \tau)$ are called \emph{strongly invertible} and \emph{strongly negative amphichiral}, respectively.

The search for sliceness notions for strongly invertible and strongly negative amphichiral knots dates back to the nineteen-eighties. In his influential article \cite{Sak86}, Sakuma introduced the notion of \emph{equivariant sliceness} for strongly invertible knots, by requiring them to bound equivariant disks smoothly and properly embedded in $B^4$. Cochran and Kawauchi first observed the \emph{$\Q$-sliceness} for the figure-eight knot and the $(2,1)$-cable of the figure-eight knot, respectively, in the sense that these knots bound disks smoothly properly embedded in some $\Q$-homology $4$-balls. Kawauchi's observation appeared in an unpublished note \cite{Kaw80}, and Cochran used the earlier work of Fintushel and Stern \cite{FS84}, which was never published. Later, Kawauchi \cite{Kaw09} proved his famous characterization result, showing that every strongly negative amphichiral knot is $\Q$-slice. See {\sc\S}\ref{sec:applications} for more details.

The main objective of this article is to merge the concepts of equivariant sliceness and $\Q$-sliceness and to introduce the study of equivariant $\mathbb{Q}$-sliceness. We call a strongly invertible knot $(K, \rho)$ \emph{equivariant $\Q$-slice} if $K$ bounds an equivariant disk smoothly properly embedded in a $\Q$-homology $4$-ball (see {\sc\S}\ref{sec:equivariant-slice}).

\subsection{Fundamentals}
\label{sec:fundamentals}

Combining these two notions leads us to study the relations between knot symmetries and introduce a specific type of symmetry, which we call \emph{Klein amphichirality}. A \emph{Klein amphichiral} knot $K \subset S^3$ is a triple $(K, \rho, \tau)$ such that the strongly invertible involution $\rho$ and strongly negative amphichiral involution $\tau$ commute with each other. This explains the motivation behind the name (see {\sc\S}\ref{sec:equivariant-slice} for more details), since the maps $\rho$ and $\tau$ together generate the \emph{Klein four group} (\emph{Vierergruppe} in German) in the symmetry group of the knot: $$V = \Z / 2\Z \times \Z / 2\Z \ \leq \ \mathrm{Sym} (S^3, K) .$$

Our first theorem provides a refinement of Kawauchi's characterization of the $\Q$-sliceness of strongly negative amphichiral knots. Kawauchi constructed a $\Q$-homology $4$-ball $Z_K$ in which $(K,\tau)$ is slice a priori depending on the given knot. However, Levine \cite{Lev23} surprisingly proved that such manifolds $Z_K$ are actually all diffeomorphic to a single $\Q$-homology ball $Z$, called the \emph{Kawauchi manifold}. Moreover, our theorem also extends Levine's crucial result to the equivariant case, showing that the pairs $(Z,\rho_K)$, corresponding to a $\Q$-homology $4$-ball for a Klein amphichiral knot $(K, \rho, \tau)$, is also unique.

\begin{introthm}
\label{thm:mainthm1}
Let $(K, \rho, \tau)$ be a Klein amphichiral knot. Then $(K, \rho)$ is equivairant $\Q$-slice. In fact, $(K, \rho)$ bounds a slice disk in the Kawauchi manifold $Z$ which is invariant under an involution $\rho_K$ of $Z$, extending $\rho$. Moreover, the involution $\rho_K$ does not depend on $(K,\rho,\tau)$ up to conjugacy in $\operatorname{Diff}(Z)$, where $\operatorname{Diff}(Z)$ denotes the group of diffeomorphisms of $Z$.
\end{introthm}

Inspired by the earlier work of Cochran, Franklin, Hedden, and Horn \cite{CFHH13}, we provide a fundamental obstruction, which was already shown in \cite{DP23b} to obstruct equivariant sliceness. It is an equivariant and rational version of the Fox-Milnor condition, which plays a key role in comparing the notions of $\Q$-sliceness and equivariant $\Q$-sliceness. It is sufficient to obstruct well-known strongly invertible $\Q$-slice knots being equivariant $\Q$-slice, including the figure-eight knot, (see {\sc\S}\ref{sec:obstructions} for more details).

\begin{introthm}
\label{thm:mainthm2}
If $(K,\rho)$ is equivariant $\Q$-slice, then its Alexander polynomial $\Delta_K(t)$ is a square.
\end{introthm}

Note that every Klein amphichiral knot is clearly \emph{strongly positive amphichiral}, i.e., isotopic to its mirror image via an involution. Furthermore, it is well known from the classical result of Hartley and Kawauchi \cite{HK79} that the Alexander polynomial of a strongly positive amphichiral knot is square. Therefore, it is interesting to compare Theorem~\ref{thm:mainthm1} with Theorem~\ref{thm:mainthm2}, see Problem~\ref{prob:+amphichiral}.

\subsection{Applications} 
\label{sec:applications}

Using our main theorems in {\sc\S}\ref{sec:fundamentals}, we further investigate the natural maps between concordance groups. To do so, we introduce the \emph{equivariant $\Q$-concordance group} $\CTQ$ (see {\sc\S}\ref{sec:equivariant-concordance}) by analyzing the equivariant $\Q$-concordance classes of strongly invertible knots.

Recall that two knots in $S^3$ are said to be \emph{concordant} if they cobound a smoothly properly embedded annulus in $S^3 \times [0,1]$. The set of oriented knots modulo concordance forms a countable abelian group, namely the \emph{concordance group} $\C$, under the operation induced by connected sum. The trivial element in $\C$ is formed by the concordance class of the unknot. The knots lying in this concordance class are the so-called \emph{slice knots}, and they bound smoothly properly embedded disks in $B^4$. The concordance group and the notion of sliceness were defined in the seminal work of Fox and Milnor \cite{FM66}. Since then, they have been very central objects of active research in knot theory and low-dimensional topology, see the survey articles \cite{Liv05, Hom17, Sav24}. Cha's monograph \cite{Cha07} systemically elaborated the \emph{$\Q$-concordance} of knots and the \emph{$\Q$-concordance group} $\CQ$. Recently, there has been a great deal of interest in these concepts as well, see \cite{KW18, Mil22, HKPS22, Lev23, Lee24}.

In \cite{Sak86}, Sakuma introduced the \emph{equivariant concordance group} $\CT$ by studying strongly invertible knots under equivariant connected sum. It is again known to be countable, but until the recent work of the first author \cite{DP23} and his joint work with Framba \cite{DPF23}, the structure of $\CT$ was completely mysterious. Unlike $\C$, it turns out that $\CT$ is non-abelian and in fact non-solvable. The comprehensive study of $\CT$ and its invariants have been the subjects of various recent articles \cite{Wat17, BI22, DMS23, HHS23, MP23}.

Considering four concordance groups and their natural maps, we therefore have the following commutative diagram. Since two concordant (resp. equivariant concordant) knots are $\Q$-concordant (resp. equivariant $\Q$-concordant), we have the surjective maps $\psi$ and $\Psi$. Moreover, $\mathfrak{f}$ and $\mathfrak{f}_\Q$ are both \emph{forgetful} maps, forgetting the additional structures.
\[
\begin{tikzcd}
\CT \arrow{r}{\Psi} \arrow[swap]{d}{\mathfrak{f}} & \CTQ \arrow{d}{\mathfrak{f}_\Q} \\%
\C \arrow{r}{\psi}& \CQ
\end{tikzcd}
\]

Using the new characterization in Theorem~\ref{thm:mainthm1}, we show that the kernel of the map $\Psi$ has a rich algebraic structure. Our constructions use certain Klein amphichiral Turk's head knots $J_n = Th (3,n)$, which are defined as the braid closures of the $3$-braids $(\sigma_1 {\sigma_2}^{-1})^n$ (see \cite{DPS24} and {\sc\S}\ref{sec:constructions} for more details). For the obstructions, we rely on the moth polynomials introduced in the recent work of Framba and the first author \cite{DPF23b} and on the application of Milnor invariants to equivariant concordance, as described in \cite{DPF23}.

\begin{introthm}
\label{thm: mainthm3}
There exists a nonabelian subgroup $\mathcal{J}$ of $\mathrm{Ker}(\Psi)$ such that its abelianization $\mathcal{J}^{ab}$ is isomorphic to $\Z^\infty$. Moreover, the image of $\mathcal{J}$ under $\mathfrak{f}$ is a $2$-torsion subgroup of $\C$.
\end{introthm}

Using the fundamental obstruction in Theorem~\ref{thm:mainthm2}, we are able to prove another interesting result about equivariant knots. Here, the constructive part follows from Cha's result (see {\sc\S}\ref{sec:obstructions}).

\begin{introthm}
\label{thm: mainthm4}
There exists a subgroup $\mathcal{K}$ of $\Ker (\mathfrak{f}_\Q)$ which surjects onto $(\Z/2\Z)^\infty$.
\end{introthm}

Previously, the other maps in the above diagram have been studied extensively. In \cite{Cha07}, Cha showed that $\mathrm{Ker}(\psi)$ has a $(\Z/2\Z)^\infty $ subgroup, generated by non-slice amphichiral knots, containing the figure-eight knot (see {\sc\S}\ref{sec:obstructions}). Recently, Hom, Kang, Park, and Stoffregen \cite{HKPS22} proved that $(2n-1, 1)$-cables of the figure-eight knot for $n \geq 2$ generate a $\Z^\infty $ subgroup in $\mathrm{Ker}(\psi)$.

On the other hand, Livingston \cite{Liv83} (cf. \cite{Kim23}) proved that the map $\mathfrak{f}$ is not surjective, exhibiting knots which are not concordant to their reverses. This result was later improved in the work of Kim and Livingston \cite{KL22}, by showing the existence of topologically slice knots which are not concordant to their reverses. More recently, Kim \cite{Kim23} showed the existence of knots that are not $\Q$-concordant to their reverses, showing that also the map $\mathfrak{f}_\Q$ is not surjective.

Potential counterexamples to the slice-ribbon conjecture based on certain cables of $\Q$-slice knots were recently eliminated by the work of Dai, Kang, Mallick, Park, and Stoffregen \cite{DKMPS22}. The core example was the $(2,1)$-cable of the figure-eight knot $K = 4_1$, denoted by $K_{2,1}$. Their strategy for obstructing the sliceness of a knot $K$ was to show that its double branched cover bounds no equivariant homology $4$-ball, remembering the data of the branching involution. This is closely related to the doubling construction (see {\sc\S}\ref{sec:constructions}) by means of the Montesinos trick which provides the diffeomorphism $\Sigma_2 ({K}_{2,1}) \cong S^3_{+1} (K \# \overline{K})$. More precisely, this diffeomorphism identifies the branching involution on $\Sigma_2 (K_{2,1})$ with the involution on surgered manifold $S^3_{+1} (K \# \overline{K})$, induced from the strong inversion on $K \# \overline{K}$. More obstructions were recently obtained by Kang, Park, and Taniguchi \cite{KPT24}.

Both $\Q$-slice and strongly invertible knots have been the subject of interesting constructions of $3$- and $4$-manifolds. For example, $\Q$-slice knots were used to exhibit Brieskorn spheres, which bound $\Q$-homology $4$-balls but not $\Z$-homology $4$-balls, see the articles by Akbulut and Larson \cite{AL18} and the second author \cite{Sav20}. Another instance was the work of Dai, Hedden, and Mallick \cite{DHM23}, which used strongly invertible slice knots to produce infinite families of new corks. More recently, Dai, Mallick, and Stofferegen \cite{DMS23} provided a new detection of exotic pairs of smooth disks in $B^4$ by relying on strong inversions of slice knots in $S^3$.

\subsection{Open Problems}

Finally, we would like to list some basic open problems for the future study of the new group $\CTQ$ and the behavior of its elements. The first problem aims to measure the difference between Theorem~\ref{thm:mainthm1} and Theorem~\ref{thm:mainthm2}.

\begin{introprob}
\label{prob:+amphichiral}
Is every equivariant $\Q$-slice knot equivariant concordant to a Klein amphichiral knot?
\end{introprob}

The second problem concerns the structure of $\CTQ$, and we expect affirmative answers to both.

\begin{introprob}
Is $\CTQ$ non-abelian? Is $\CTQ$ non-solvable?
\end{introprob}

The other problem is about the potential complexity of equivariant $\Q$-slice knots. Following the paper of Boyle and Issa \cite{BI22}, recall that given a strongly invertible knot $(K, \rho)$, the \emph{equivariant $4$-genus} $\widetilde{g_4}$ of $K$ is the minimal genus of an orientable, smoothly properly embedded surface $S \subset B^4$ with boundary $K$ for which $\rho$ extends to an involution $\widetilde{\rho}: (B^4, S) \to (B^4, S)$. Previously, using Casson-Gordon invariants, Miller \cite{Mil22} proved that there are $\Q$-slice knots with arbitrarily large $g_4$, namely the classical smooth $4$-genus.

\begin{introprob}
Are there equivariant $\Q$-slice knots with arbitrarily large $\widetilde{g_4}$?
\end{introprob}

Theorem \ref{thm:mainthm2} provides a first obstruction to equivariant $\Q$-sliceness, which can be seen as a first \emph{algebraic concordance} obstruction in this setting. We would like to ask whether it is possible to define other obstructions and invariants, similar to the ones obtained in \cite{Lev69b,Lev69,Cha07,DP23b}. See also Remark~\ref{rem:algebraic concordance}.

\begin{introprob}
\label{prob:algebraic concordance}
Can we define a notion of equivariant algebraic $\Q$-concordance?
\end{introprob}

The knot Floer theoretic invariants have had several important applications to the study of knot concordance, see Hom's surveys \cite{Hom17,Hom23} for more details. The two famous invariants -- $\tau$ and $\epsilon$ -- are also known to be $\Q$-concordance invariants. More recently, Dai, Mallick, and Stoffregen \cite{DMS23} also provided several equivariant concordance invariants using knot Floer homology. As a final problem, we would like to ask:

\begin{introprob}
\label{prob:floer_homology}
Can we define equivariant $\Q$-concordance invariants using knot Floer homology?
\end{introprob}

\subsection*{Organization} In {\sc\S}\ref{sec: equivariant}, we study equivariant $\Q$-concordances in the broad perspective. We review symmetries of knots and introduce the notion of Klein amphichirality in {\sc\S}\ref{sec:symmetry}. Then, in {\sc\S}\ref{sec:equivariant-slice}, we prove Theorem~\ref{thm:mainthm1}. Next, we introduce the equivariant $\Q$-concordance group in {\sc\S}\ref{sec:equivariant-concordance}. In {\sc\S}\ref{sec:constructions}, we construct equivariant $\Q$-slice knots by using Klein amphichiral Turk's head knots. Finally, we prove Theorem~\ref{thm:mainthm2} in {\sc\S}\ref{sec:obstructions}, and give examples of non-equivariant $\Q$-slice knots. We close the section by proving Theorem~\ref{thm: mainthm4}. In {\sc\S}\ref{sec:independence}, we particularly work on the obstructions. After discussing preliminary notions such as weighted graphs, Gordon-Litherland forms, and moth polynomials, we show independence of certain equivariant $\Q$-slice knots in $\CT$. We finally prove Theorem~\ref{thm: mainthm3}. 

\subsection*{Acknowledgements} Our project has started during the events Winter Braids XIII in Montpellier and Between the Waves 4 in Besse, so we would like to thank the organizers of both meetings. We are grateful to Leonardo Ferrari for the useful discussions on hyperbolic knot symmetries.

\section{Equivariant Q-Sliceness and Q-Concordance}
\label{sec: equivariant}

\subsection{Symmetries of Knots and Klein Amphichirality}
\label{sec:symmetry}
Following Kawauchi's book \cite[{\sc\S}10]{Kaw90}, we consider the following important symmetries of knots. Furthermore, we use the resolution of the Smith conjecture \cite{W69, MB84} to identify the fixed point set of a given involution, denoted by $\mathrm{Fix} (\cdot)$. We also set the notation $\operatorname{Diff}(\cdot)$ (resp. $\operatorname{Diff}^+(\cdot)$) for the group of diffeomorphisms (resp. orientation-preserving diffeomorphisms) of a manifold, and $\operatorname{Diff}^-(\cdot) \doteq \operatorname{Diff}(\cdot) \setminus \operatorname{Diff}^+(\cdot)$, i.e., the set of orientation-reversing diffeomorphisms of a manifold. 

\begin{defn}
A knot $K$ in $S^3$ is said to be:
\begin{itemize}[leftmargin=2em]
    \item \emph{invertible} if there is a map $\rho \in \operatorname{Diff}^+(S^3)$ such that $\rho (K) = -K$. If $\rho$ is further an involution, then $(K, \rho)$ is called \emph{strongly invertible}. In this case, we have $\mathrm{Fix}(\rho) = S^1$ and $\mathrm{Fix}(\rho) \cap K = S^0$. Moreoever, the knots $(K, \rho)$ and $(K', \rho')$ are called \emph{equivalent} if there is a map $f \in \operatorname{Diff}^+(S^3) $ such that $f(K) = K'$ and $f \circ \rho \circ f^{-1} = \rho^{-1}.$

    \item \emph{negative amphichiral} if there is a map $\tau \in \operatorname{Diff}^-(S^3)$ such that $\tau (K) = -K$. If $\tau$ is further an involution, $(K, \tau)$ is called \emph{strongly negative amphichiral}. In this case, we have either $\mathrm{Fix}(\tau) = S^0$ or $\mathrm{Fix}(\tau)=S^2$. 
    
    \item \emph{positive amphichiral} if there is a map $\delta \in \operatorname{Diff}^-(S^3)$ such that $\delta(K) = K$. If $\delta$ is further an involution, then $(K, \delta)$ is called \emph{strongly positive amphichiral}. In this case, we have $\mathrm{Fix}(\tau) = S^0$. 

    \item \emph{$n$-periodic} if there is a map $\theta \in \operatorname{Diff}^+(S^3)$ such that $\theta (K) = K$, $\Fix(\theta)\cap K=\emptyset$ and $\theta$ is period of $n$, i.e., $n$ is the minimal number so that $\theta^n = \mathrm{id} \in \operatorname{Diff}^+(S^3) $. If $\mathrm{Fix}(\theta) = S^1$ (resp. $\Fix(\theta)=\emptyset$), then we say that $(K,\theta)$ is \emph{cyclically periodic} (resp. \emph{freely periodic}).
\end{itemize}
\end{defn}

\begin{figure}[htbp]
\centering
\includegraphics[width=0.6\columnwidth]{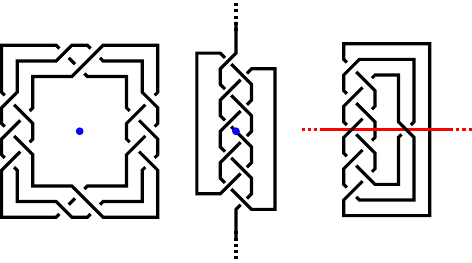}
\caption{From left to right: strongly positive amphichiral, strongly negative amphichiral and strongly invertible symmetries on $10_{123}$. In the first two cases, the involution is given by the $\pi$-rotation around the blue dot composed with the reflection along the plane of the diagram. The third symmetry is given by the $\pi$-rotation around the red axis.}
\label{fig:Th(3,5)}
\end{figure} 

In order to compare the symmetries of knots in our new context, we introduce the following crucial notion.

\begin{defn}
\label{defn:klein}
A knot $K$ in $S^3$ is said to be \emph{Klein amphichiral} if there exist two involutions $\rho,\tau: S^3 \to S^3$ such that
\begin{itemize}[leftmargin=2em]
    \item $(K,\rho)$ is strongly invertible,
    \item $(K,\tau)$ is strongly negative amphichiral,
    \item $\rho \circ \tau = \tau \circ \rho$.
\end{itemize}
\end{defn}

The \emph{symmetry group} of a knot $K$ in $S^3$, which is denoted by $\mathrm{Sym} (S^3, K)$, is defined to be the mapping class group of the knot exterior $S^3 \setminus \nu(K)$, see \cite[{\sc\S}10.6]{Kaw90}. Denote by $\mathrm{Sym}^+(S^3,K)$ the subgroup consisting of orientation-preserving maps. Observe that since the maps $\tau$ and $\rho$ of a Klein amphichiral knot commute, they together generate the \emph{Klein four group} $$ V = \Z/2\Z \times \Z/2\Z ,$$  hence the composition map $\tau \circ \rho$ is a strongly positive amphichiral involution for $K$. See Figure~\ref{fig:10_123_sym} for an example of a Klein amphichiral knot.

\begin{figure}[ht]
    \centering
    \includegraphics[width=0.7\linewidth]{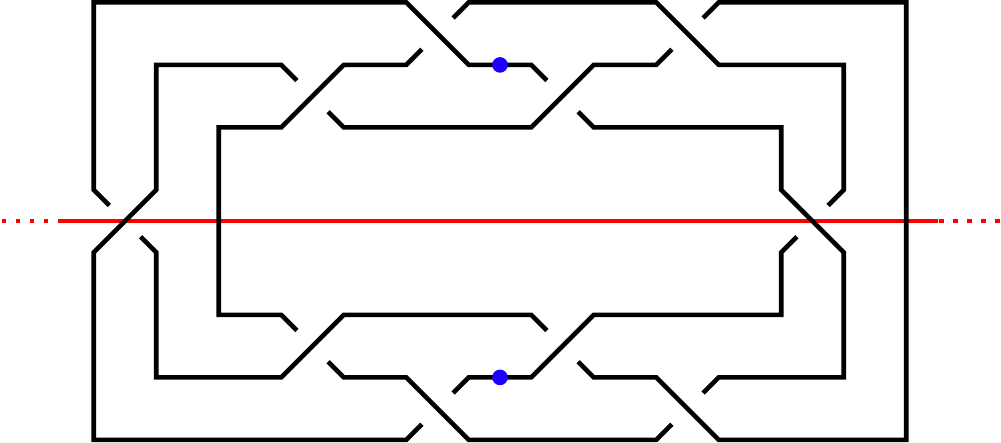}
    \caption{Klein amphichiral symmetry on $10_{123}$: $\rho$ is given by a $\pi$-rotation around the red axis, while $\tau$ is the point reflection around the two blue points.}
    \label{fig:10_123_sym}
\end{figure}

Notice that there exist two distinct types of Klein amphichiral symmetries, distinguished by the fixed point set of $\tau$: \begin{enumerate}
    \item $\Fix(\tau)\cong S^2$, which forces $K=J\widetilde{\#}J^{-1}$ for some DSI knot $J$,
    \item $\Fix(\tau)\cong S^0$.
\end{enumerate}

In the first case, $K$ is clearly equivariant slice, which we will regard as the trivial case. Therefore, we will always assume that the symmetry falls in the second case.

\begin{rem}
    The recent work of Boyle, Rouse, and Williams \cite{BRW23} provided a classification result for symmetries of knots in $S^3$. In their terminology, a Klein amphichiral knot corresponds to a $D_2$-symmetric knot of type SNASI-(1).
\end{rem}

A knot $K$ is called \emph{hyperbolic} if the knot complement $S^3 \setminus \nu(K)$ admits a complete metric of constant curvature $-1$ and finite volume. Now, we recall the following proposition relating the symmetries of hyperbolic knots with their symmetry groups.

\begin{prop}
\label{prop:equivariant}
Let $K$ be a hyperbolic, invertible, and negative amphichiral knot in $S^3$. Suppose that $\mathrm{Sym} (S^3, K) = D_{2m}$ with $m$ odd, where $D_n$ is the dihedral group with $2n$ elements. Then $K$ admits a unique strongly invertible involution $\rho$ (up to \emph{equivalence}, see {\sc\S}\ref{sec:constructions}) and a strongly negative amphichiral involution $\tau$ such that $(K,\rho,\tau)$ is Klein amphichiral.
\end{prop}
\begin{proof}

Let $K$ be a hyperbolic, invertible and negative amphcheiral knot in $S^3$. Since $K$ is a hyperbolic knot, by the work of Kawauchi \cite[Lemma~1]{Kaw79}, we have a pair of strongly negative amphichiral and strongly invertible involutions $\tau, \rho \in \mathrm{Sym} (S^3, K)$ for the knot $K$. As $K$ is both strongly negative amphichiral and strongly invertible, we know that $\mathrm{Sym} (S^3, K) = D_{2m}$ for some $m$, due to the result of Kodama and Sakuma \cite[Lemma~1.1]{KS92}.

Now, assume that $\mathrm{Sym} (S^3, K) = D_{2m}=\langle s,t\;|\;t^{2m}=s^2=1, sts=t^{-1}\rangle$ with $m$ odd. It is a well-known fact that the involutions of $D_{2m}$ split into three conjugacy classes, namely $\{t^m\}$, $\{st^{2i}\}_{0\leq i< m}$ and $\{st^{2i+1}\}_{0\leq i<m}$. Since $\tau$ and $\rho$ are, respectively, orientation reversing and preserving, they are not conjugate. If one of them corresponds to $t^m$, which is central, then they commute. Otherwise, they lie respectively in $\{st^{2i}\}_{0\leq i\leq m}$ and $\{st^{2i+1}\}_{0\leq i\leq m}$ (or vice versa). By changing $\rho$ and $\tau$ with some conjugates, we can suppose $\tau=st$ and $\rho=st^{2i}$ so that $2i\equiv 1\mod{m}$ (which exists since $m$ is odd). It is now easy to check that $\rho \circ \tau = \tau \circ \rho = t^m$, where $t^m$ is a strongly positive amphichiral involution. Therefore, $(K,\rho,\tau)$ is Klein amphichiral.  
\end{proof}

Now we fix a standard model for the Klein amphichiral symmetry. To do so, we can think of $S^3=\{(z,w)\in\mathbb{C}^2\;|\;|z|^2+|w|^2=1\}$. Now consider $\rho$ and $\tau$ to be the following involutions
\begin{align*}
    \rho: S^3 \to S^3, \quad (z,w)\mapsto(-z,w), \\
    \tau:S^3 \to S^3, \quad (z,w)\mapsto(\overline{z},-w).
\end{align*} so that $$ (\rho \circ \tau) (z,w) = (-\overline{z}, -w) = ( \tau \circ \rho) (z,w) .$$

According to \cite{BRW23}, up to conjugation in $\operatorname{Diff}^+(S^3)$, we can always suppose that the Klein amphichiral symmetry is given by the action of $\rho$ and $\tau$ above. Then the fixed point sets of these involutions are given by
\begin{align*}
    \Fix(\rho)&=\{(0,w)\;|\;w\in S^1\},\\
    \Fix(\tau)&=\{(\pm 1,0)\},\\
    \Fix(\rho \circ \tau)&=\{(\pm i,0)\}.
\end{align*}

Let $N_{\rho}$, $N_\tau$, and $N_{\rho \circ \tau}$ be small equivariant tubular neighbourhoods of $\Fix(\rho)$, $\Fix(\tau)$, and $\Fix(\rho\circ \tau)$, respectively. Then we have:
\begin{itemize}[leftmargin=2em]
    \item $N_\rho\cong D^2\times S^1$ and $\rho_{|N_\rho}=(-\id_{D^2},\id_{S^1})$,
    \item $N_\rho\cap K=\{(t\cdot z_i,w_i)\;|\;t\in[-1,1]\}$ for some $z_0,z_1,w_0,w_1\in S^1$,
    \item $N_\tau\cong B^3_1\sqcup B^3_{-1}$, where $B^3_{\pm 1}$ is a small ball centered at $\pm 1$ and $\tau_{|B^3_{\pm 1}}=-\id$,
    \item $B_{\pm 1}^3\cap K=\{t\cdot p\;|\;t\in [-1,1]\}$ for some $p\in S^2$,
    \item $N_{\rho \circ \tau}\cap K=\emptyset$.
\end{itemize}

Let $Y=S^3\setminus \operatorname{int}(N_\rho\cup N_\tau\cup N_{\rho \circ \tau})$. Observe that $Y$ is simply a solid torus with four small balls removed and that $K\cap Y$ consists of four arcs: two of the arcs connect $\partial N_\rho$ to $\partial B^3_1$ and the other two connect $\partial N_\rho$ to $\partial B^3_{-1}$, see Figure~\ref{fig:ext_fix}.

\begin{figure}[ht]
    \centering
    \includegraphics[width=0.4\linewidth]{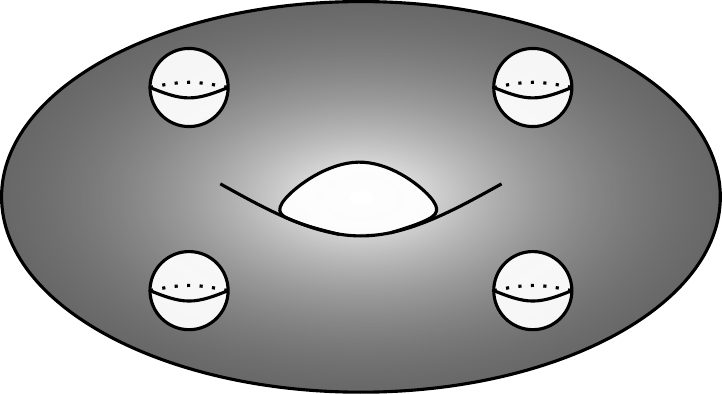}
    \caption{The manifold $Y$, given by removing the fixed point sets.}
    \label{fig:ext_fix}
\end{figure}

Define $\overline{Y}=Y/V$ to be the quotient of $Y$ by the action of $V$. It is not difficult to see that $\overline{Y}$ is a non-orientable $3$-manifold with boundary, whose boundary components are given as follows:In particular, one can see that $\overline{Y}\cong (\mathbb{RP}^2\times I)\natural(\mathbb{RP}^2\times I)$. 
\begin{itemize}[leftmargin=2em]
    \item a Klein bottle, which is the image of the toric boundary component of $Y$, denoted by $A$,
    \item two $\mathbb{RP}^2$'s, which are the images of $\partial N_\tau$ and $\partial N_{\rho \circ \tau}$, respectively. Denote the image of $N_\tau$ by $B$.
\end{itemize} 
In particular, one can see that $\overline{Y}\cong (\mathbb{RP}^2\times I)\natural(\mathbb{RP}^2\times I)$.
Now, denote the image of $K\cap Y$ in $\overline{Y}$ by $\overline{K}$. Then $\overline{K}$ is a simple properly embedded arc in $\overline{Y}$ starting at $A$ and ending in $B$, see Figure~\ref{fig:ext_fix_quo}.

\begin{figure}[ht]
    \centering
    \includegraphics[width=0.3\linewidth]{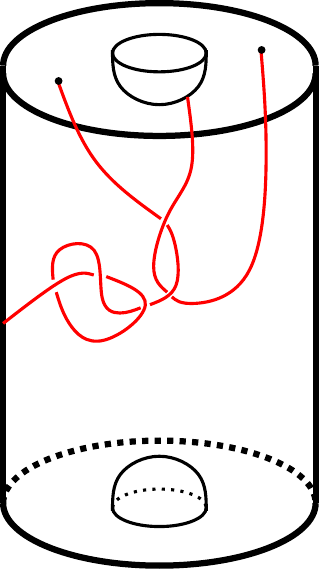}
    \caption{The quotient manifold $\overline{Y}$. Both the upper and lower punctured disks are glued to themselves by $-\id$. In red, there is an example of an arc $\overline{K}$ going from the Klein bottle boundary component to one $\mathbb{RP}^2$ boundary component.}
    \label{fig:ext_fix_quo}
\end{figure}

Define now $$\pi_1(\overline{Y},A,B) = \{ \gamma:[0,1]\to\overline{Y} \ \vert \ \gamma(0)\in A, \ \gamma(1)\in B \} / \text{homotopy}.$$ Notice that every class in $\pi_1(\overline{Y},A,B)$ can be represented by a simple properly embedded arc.
Observe that any homotopy between two properly embedded arcs in $\overline{Y}$ connecting $A$ to $B$ is lifted to an equivariant homotopy in $Y$. This can be closed up to an equivariant homotopy between the corresponding Klein amphichiral knots in $S^3$ in the following way.
Let $\gamma_t$ be the lift of the homotopy at time $t$. Then $\gamma_t$ is given by four proper arcs in $Y$ invariant under the action of $V$. We extend the homotopy over $N_\rho$ and $N_\tau$ following the local models described above. Observe that $\gamma_t\cap\partial N_\rho$ is given by $4$ points: $(z_0,w_0), (-z_0,w_0), (z_1,w_1), (-z_1,w_1)\in S^1\times S^1$. We connect then the components of $\gamma_t$ by adding the arcs $(s\cdot z_0,w_0)$ and $(s\cdot z_1,w_1)$, $s\in [-1,1]$ inside $N_\rho$.

Similarly, consider the intersection $\gamma_t\cap \partial N_\tau$, which is given by four points $p,-p\in \partial B_{+1}^3$ and $q,-q\in\partial B_{-1}^3$. Then the components of $\gamma_t$ are connected by adding the arcs $\{s\cdot p\}_{s\in [-1,1]}\subset B_{+1}^3$ and $\{s\cdot q\}_{s\in [-1,1]}\subset B_{-1}^3$. In this way, we obtain an equivariant homotopy of Klein amphichiral knots in $S^3$.

We can sum up the discussion above in the following lemma:
\begin{lem}
The natural map
$$
\{\text{Klein amphichiral knots}\}/\text{equivariant homotopy}\longrightarrow\pi_1(\overline{Y},A,B)$$
is a bijection.
\end{lem}

In order to prove the refinement of Levine's theorem given in Theorem \ref{thm:mainthm1}, the essential ingredient is the following lemma.

\begin{lem}\label{lem:equiv_crossing}
    Every Klein amphichiral knot can be turned into the standard Klein amphichiral unknot by a finite number of $V$-equivariant crossing changes.
\end{lem}

\begin{proof}
Let $K_0$ and $K_1$ be two Klein amphichiral knots, and let $\overline{K_0}$ and $\overline{K_1}$ be the corresponding arcs in $\overline{Y}$. Suppose that there exists a homotopy $\overline{K}_t$, $t\in[0,1]$ between $\overline{K_0}$ and $\overline{K_1}$. Then up to a small perturbation, we can suppose that $\overline{K_t}$ is a simple properly embedded arc connecting $A$ to $B$ for all $t\in[0,1]$ except for finitely many times, for which $\overline{K_t}$ has a point of double transverse self-intersection.

Let $K_t$, $t\in[0,1]$ be the corresponding homotopy between $K_0$ and $K_1$. Then for all $t$, we have that $K_t$ is a Klein amphichiral knot, except at finitely many times, for which it has some points of double transverse self-intersection. Such double points either arise from the double points of $\overline{K_t}$ or they might come from the closing up procedure of the homotopy described above. Using the notation in the discussion above, for example, it might happen that $w_0=w_1$, leading to a new double point inside $N_\rho$ during the homotopy.

Thus, it suffices to show that $\pi_1(\overline{Y},A,B)$ consists of a single point. For this purpose, pick two points $p\in A$ and $q\in B$, and consider the set
$$
\pi_1(\overline{Y},p,q)=\{\gamma:[0,1]\to\overline{Y}\;|\;\gamma(0)=p, \ \gamma(1)=q\}/\text{homotopy}.
$$
Notice that since $A$ and $B$ are connected, the natural map
$$
\iota: \pi_1(\overline{Y},p,q)\to\pi_1(\overline{Y},A,B)
$$
induced by the inclusion is surjective.
Observe then that 
\begin{itemize}[leftmargin=2em]
    \item both $\pi_1(A,p)$ and $\pi_1(B,q)$ act on $\pi_1(\overline{Y},p,q)$ by pre-composition and post-composition, respectively,
    \item the image of a class under $\iota$ is not changed by the action of $\pi_1(A,p)$ and $\pi_1(B,q)$.
\end{itemize}

\noindent This reduces our argument to show that the combined action of $\pi_1(A,p)$ and $\pi_1(B,q)$ is transitive on $\pi_1(\overline{Y},p,q)$, which is a consequence of the following two facts: 
\begin{itemize}[leftmargin=2em]
    \item $\pi_1(\overline{Y},p)$ acts transitively on $\pi_1(\overline{Y},p,q)$ by pre-composition,
    \item $\pi_1(\overline{Y})=\langle\pi_1(A),\pi_1(B)\rangle$.
\end{itemize}

\noindent Therefore, $\pi_1(\overline{Y},A,B)=\{*\}$, as we desired.
\end{proof}

\subsection{The Equivariant Q-Slice Knots}
\label{sec:equivariant-slice}

An \emph{equivariant $\Q$-slice slice knot} is a strongly invertible knot $(K, \rho_K )$ in $S^3$ that bounds a disk $D$ smoothly properly embedded in a $\Q$-homology $4$-ball (a $4$-manifold having the $\Q$-homology of $B^4$) with an orientation preserving involution $\rho : W \to W$ such that $$ \partial D = K, \ \ \  \rho_{\vert_{\partial W = S^3}} = \rho_K, \ \ \  \text{and} \ \ \  \rho (D) = D ,$$ cf. the conditions (\ref{defn:p3}) and (\ref{defn:p4}) below.

Recall that Kawauchi's famous result \cite[{\sc\S}2]{Kaw09} provided a characterization for $\Q$-slice knots, showing that every strongly negative amphichiral knot is $\Q$-slice. Now, we prove an equivariant refinement of Kawauchi's theorem by proving that every Klein amphichiral knot is equivariant $\Q$-slice. This is essentially the first half of Theorem~\ref{thm:mainthm1}.

\begin{proof}[Proof of Theorem \ref{thm:mainthm1}, the $1^{st}$ half] 

Consider the product extension of $\rho$ and $\tau$ on $S^3\times[0,1]$ , which we will still denote them using the same letters.

Let $X_K$ be the $4$-manifold obtained from $S^3\times [0,1]$ by attaching a $0$-framed $2$-handle along $K\subset S^3\times\{1\}$. We can find a $\langle \rho,\tau\rangle$-invariant neighbourhood $N(K)\cong S^1\times D^2$ of $K$ such that
\begin{itemize}[leftmargin=2em]
    \item $\rho(z,w)=(\overline{z},\overline{w})$, for $(z,w)\in S^1\times D^2$,
    \item $\tau(z,w)=(\overline{z},-w)$ for $(z,w)\in S^1\times D^2$.
\end{itemize}

Therefore, we can extend $\rho$ and $\tau$ over the $2$-handle $D^2\times D^2$ by using the obvious extensions of the formulas above. Denote the core disk of the $2$-handle by $D=D^2\times \{0\}$. Observe in particular that $\rho(D)=D$.

Let $S^3_0(K)$ be the $0$-surgery of $K$. Notice that the restriction of $\tau$ on the $S^3_0(K)$ component of $\partial X_K$ is fixed-point free and orientation-reversing. Therefore, we can glue $X_K$ to itself by identifying $x\sim\tau(x)$ for all $x\in S^3_0(K)$, and obtain an oriented $4$-manifold $Z_K$ with $\partial Z_K = S^3$, namely the Kawauchi manifold.

Let $D'$ be the disk obtained by gluing the product cylinder $K\times [0,1]$ with $D$. In \cite[Lemma~2.3, Theorem~1.1]{Kaw09}, Kawauchi proves that $D'$ is a slice disk for $K$ in $Z_K$, and $Z_K$ is a $\Q$-homology $4$-ball.

Finally, since $\rho$ and $\tau$ commute, the action of $\rho$ on $X_K$ induces an involution $\rho_{Z_K}$ on $Z_K$ such that $\rho_{Z_K} (D') = D'$. This shows that $D'$ is actually an equivariant slice disk for $(K,\rho)$ in $Z_K$. Therefore, $(K,\rho)$ is equivariantly $\Q$-slice in the Kawauchi manifold $Z_K=Z$.
\end{proof}

Due to its construction, a priori the Kawauchi manifold $Z$ seems to depend on the strongly negative amphichiral knot. However, Levine \cite{Lev23} proved that $Z$ is a unique manifold in the sense that all strongly negative amphichiral knots bound smooth disks in $Z$. Our proof extends Levine's theorem to the equivariant case by showing that every Klein amphichiral knot $(K, \rho, \tau)$ bounds an equivariant disk in $(V, \rho_K)$, where $\rho_K$ is an involution extending $\rho$. This is the remaining part of Theorem~\ref{thm:mainthm1}.

\begin{proof}[Proof of Theorem \ref{thm:mainthm1}, the $2^{nd}$ half] 

Observe that Levine's original proof \cite{Lev23} has three main ingredients: a $5$-dimensional handlebody argument, an equivariant unknotting theorem of Boyle and Chen \cite[Proposition~3.12]{BC24}, and the equivariant isotopy extension theorem of Kankaanrinta \cite[Theorem~8.6]{Kan07}.

The first and last ingredients are essentially the same for our extended argument. In this setting, the second ingredient is simply replaced by Lemma~\ref{lem:equiv_crossing}.
\end{proof}

\subsection{The Equivariant Q-Concordance Group}
\label{sec:equivariant-concordance}

A \emph{direction} on a given strongly invertible knot $(K,\rho)$ is a choice of oriented \emph{half-axis} $h$, i.e. the choice of an oriented connected component of $\mathrm{Fix} (\rho)\setminus K$. We will call a triple $(K,\rho,h)$ a \emph{directed strongly invertible knot} (a \emph{DSI knot} in short). 

We say that two DSI knots $(K_0,\rho_0,h_0)$ and $(K_1,\rho_1,h_0)$ are \emph{equivariantly isotopic} if there exists and $\varphi\in\operatorname{Diff}^+(S^3)$ such that $\varphi(K_0)=K_1$, $\varphi\circ\rho_0=\rho_1\circ\varphi$ and $\phi(h_0)=h_1$ as oriented half-axes. Given two DSI knots $(K_0,\rho_0,h_0)$ and $(K_1,\rho_1,h_0)$, we may form \emph{equivariant connected sum} operation $\widetilde{\#}$, introduced by Sakuma \cite[{\sc\S}1]{Sak86}. This yields a potentially new DSI knot $( K_0 \widetilde{\#} K_1,\rho_0 \widetilde{\#} \rho_1, h_0\widetilde{\#}h_1 )$ whose oriented half-axis starts from the tail of the half-axis for $K_0$ and ends at the head of the half-axis for $K_1$, see Figure~\ref{fig:conn_sum}.

\begin{figure}[ht]
\centering
\begin{tikzpicture}

\node[anchor=south west,inner sep=0] at (0,0){\includegraphics[width=0.5\linewidth]{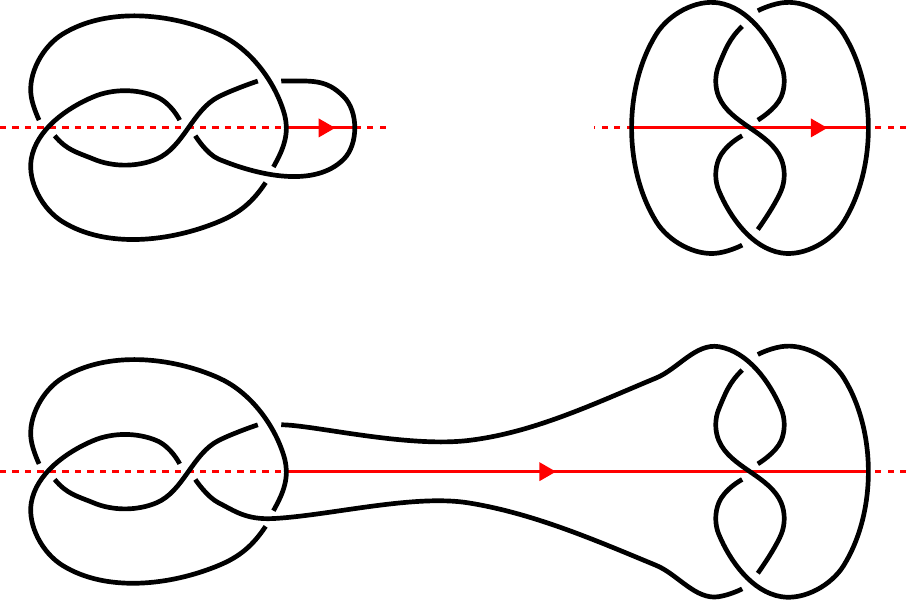}};
\draw (4.3,2.3) node[scale=2] (a) {$=$};
\draw (4.3,4.5) node[scale=2] (b) {$\widetilde{\#}$};
\end{tikzpicture}
    \caption{An example of equivariant connected sum.}
    \label{fig:conn_sum}
\end{figure}

Let $(K_0,\rho_0,h_0)$ and $(K_1,\rho_1,h_1)$ be two DSI knots in $S^3$. We call $(K_0,\rho_0,h_0)$ and $(K_,\rho_1,h_1)$ \emph{equivariant $\Q$-concordant} and denote it by $$(K_0,\rho_0,h_0) \thicksim_\Q (K_,\rho_1,h_1)$$ if there is a pair of smooth manifolds $\left ( W, A \right ) $  satisfying the following conditions: \begin{enumerate}
    \item $W$ is a $\Q$-homology cobordism from $S^3$ to itself, i.e., $H_* (W; \Q ) \cong H_* (S^3 \times [0,1]; \Q )$,
    \item $A$ is a submanifold of $W$ with $A \cong S^1 \times [0,1]$ and $\partial A \cong -(K_0) \cup K_1$,
    \item \label{defn:p3} There is an involution $\rho_W: W \to W$ that extends $\rho_0$ and $\rho_1$, and has $\rho_W (A) = A$,
    \item \label{defn:p4} Denote by $F$ the fixed-point surface of $\rho_W$. Then the half-axes $h_0,h_1$ lie in the same boundary component of $F-A$ and their orientations induce the same orientation on $F-A$.
\end{enumerate}

\begin{rem}
Observe that if $W$ as above is not a $\Z/2\Z$-homology cobordism then $\Fix(\rho_W)$ might not be an annulus. This implies that the equivariant $\Q$-sliceness for a directed strongly invertible knot is not equivalent to the equivariant $\Q$-sliceness of its \emph{butterfly link} (see Definition \ref{butterfly_link}) or \emph{moth link} (see \cite[Definition 5.2]{DPF23}), contrary to the non-rational case, as proved in \cite[Proposition 4.2]{BI22} and \cite[Proposition 5.4]{DPF23}.
As a consequence, several invariants obtained from these links, such as the butterfly and moth polynomial, do not automatically vanish for equivariant $\Q$-slice knots (compare with the computations in Section \ref{ssec:buttefly_moth}).
\end{rem}

As a natural extension of the equivariant conconcordance group $\CT$ defined by Sakuma \cite[{\sc\S}4]{Sak86} (see also Boyle and Issa \cite[{\sc\S}2]{BI22}), we introduce the \emph{equivariant $\Q$-concordance group} $\CTQ$ as $$\CTQ \doteq \left ( \left \{ \text{DSI knots in} \ S^3 \right \} / \thicksim_\Q , \ \widetilde{\#} \right ) .$$ 

\noindent We have the following main properties:

\begin{itemize}[leftmargin=2em]
    \item The operation is induced by the equivariant connected sum $\widetilde{\#}$.
    \item The identity element is the equivariant $\Q$-concordance class of the unknot $(U, \rho_U , h_U)$\footnote{The unknot $U$ in $S^3$ has a unique strong inversion, see \cite[Definition~1.1, Lemma~1.2]{Sak86}.}, see Figure~\ref{fig:unknot}.
    \item The inverse element for $(K, \rho , h)$ is given by axis-inverse of the mirror of $K$, i.e., $(\overline{K}, \rho , -h)$.
\end{itemize}

\begin{figure}[ht]
    \centering
    \includegraphics[width=0.15\linewidth]{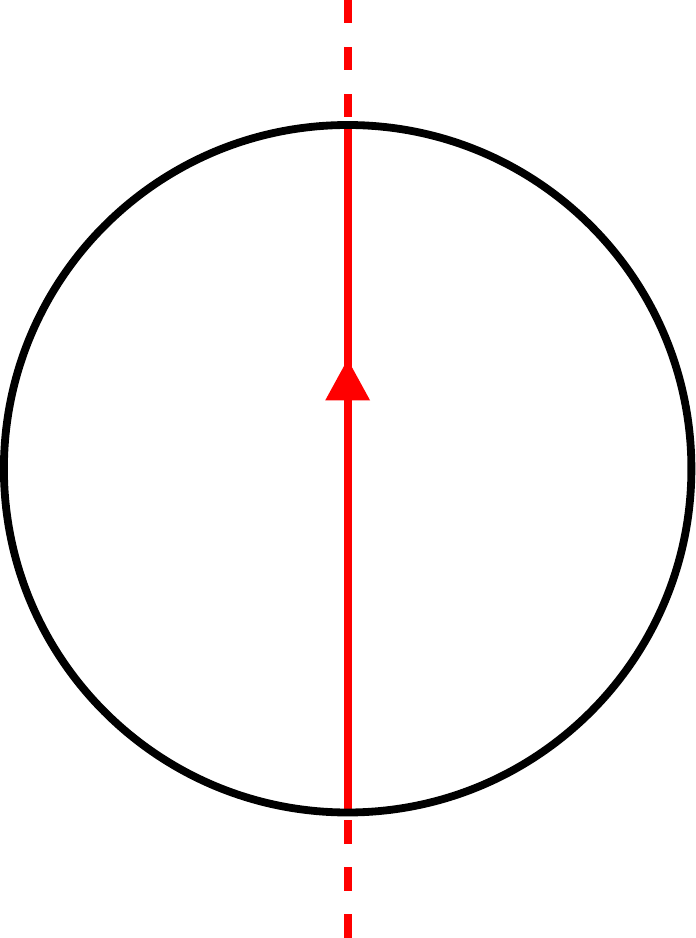}
    \caption{The DSI unknot.}
    \label{fig:unknot}
\end{figure}

\subsection{A Construction for Equivariant Q-Slice Knots}
\label{sec:constructions}

In this subsection, we construct examples of equivariant $\Q$-slice knots. We now introduce the first family of examples which are obtained by using the following doubling argument.

Given an oriented knot $K$, its \emph{double} $\mathfrak{r}(K)$ is the DSI knot obtained by $K\#r(K)$, (where $r(K)$ is the \emph{reverse} of $K$, i.e. the same knot but oppositely oriented) with the involution $\rho$ that exchanges $K$ and $r(K)$ (namely the $\pi$-rotation around the vertical axis in Figure~\ref{rK}. The direction on $\mathfrak{r}(K)$ is given as follows: the connected sum can be performed by a suitable band move along a grey band $B$ where $\Fix(\rho)\cap B$ is the half-axis $h$. Here, $h$ is oriented as the portion of $B$ lying on $K$.

\begin{figure}[ht]
\centering
\begin{tikzpicture}

\node[anchor=south west,inner sep=0] at (0,0){\includegraphics[scale=0.3]{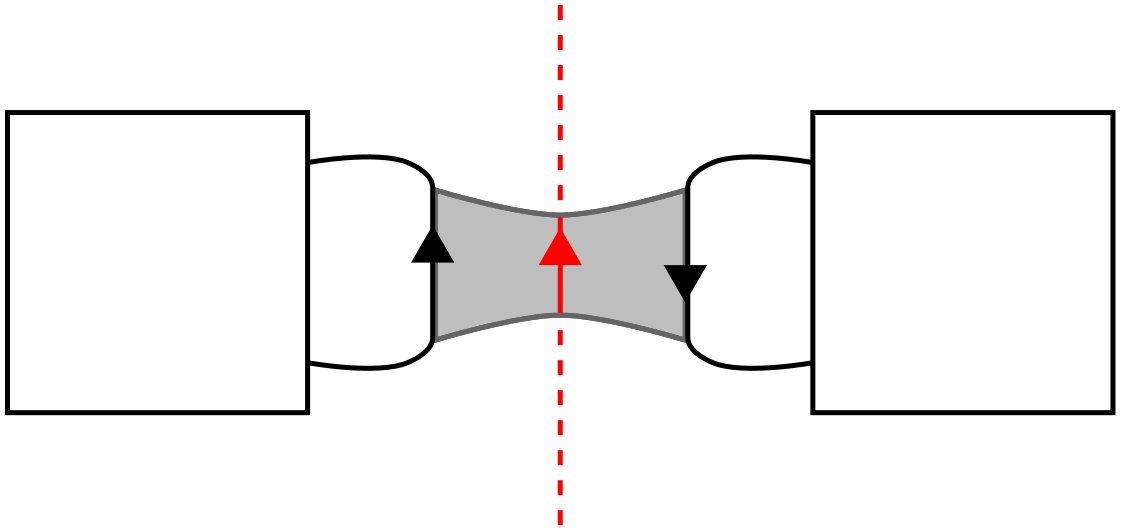}};
\node[label={$K$}] at (1.2,1.7){};
\node[label={$h$}] at (4.1,1.7){};
\node[label={$r(K)$}] at (7.8,1.7){};

\end{tikzpicture}
    \caption{The DSI knot $\mathfrak{r}(K)$ with the solid chosen half-axis.}
    \label{rK}
\end{figure}
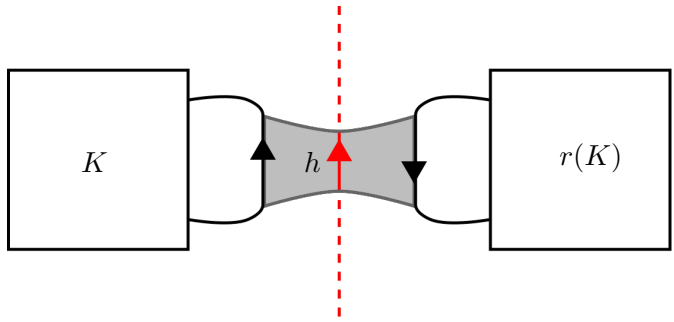

As proven by Boyle and Issa \cite[{\sc\S}2]{BI22}, $\mathfrak{r}$ defines an injective homomorphism $\mathfrak{r}:\C \longrightarrow \CT.$ The same argument actually shows that $\mathfrak{r}$ also induces a homomorphism $\mathfrak{r}_{\Q}:\CQ\to\CTQ$ (not necessarily injective),
which fits into the following commutative diagram
\begin{center}
    \begin{tikzcd} \C\ar[r,"\mathfrak{r}"]\ar[d,"\psi"]&\CT\ar[d,"\Phi"]\\
        \CQ\ar[r,"\mathfrak{r}_{\Q}"]&\CTQ.
    \end{tikzcd}
\end{center}

Therefore, the first family of DSI knots are trivial in the sense that they lie in the $\Ker (\Phi)$ and are simply given by the image of $\mathrm{Ker}(\psi)$ under $\mathfrak{r}$.

It is known that the algebraic structure of $\mathrm{Ker}(\psi)$ is very complicated. In particular, $\mathrm{Ker}(\psi)$ contains a $\Z^\infty \oplus (\Z/2\Z)^\infty$ subgroup due to the work of Cha \cite{Cha07} and Hom, Kang, Park, and Stoffregen \cite{HKPS22}.

At the moment, the structure of $\CT$ is quite mysterious. The first author proves that $\CT$ is non-abelian \cite{DP23} and together with Framba that $\CT$ is also non-solvable \cite{DPF23}. He also observes in \cite[Corollary~1.16]{DP23b} that the equivariant concordance group splits as the following direct sum
$$\CT= \mathrm{Ker} (\mathfrak{b})\oplus\mathfrak{r}(\mathcal{C}),$$
where $\mathrm{Ker} (\mathfrak{b})$ is the subgroup of $\CT$ formed by DSI knots $K$ such that their \emph{$0$-butterfly links} $L_b^0(K)$ (see Definition \ref{butterfly_link} and \cite{BI22}) have slice components (but they are not necessarily slice as links).

Now, we construct the non-trivial examples of equivariant $\Q$-slice knots in $S^3$. Let $$\J = \{ n \in \mathbb{N} \ \vert \ n>1 \ \text{and} \ n \not \equiv 0 \mod 3 \} .$$ We can write $$\J = \JE \sqcup \JO = \{ n \in \J \ \vert \ n \ \text{is even} \} \sqcup \{ n \in \J \ \vert \ n \ \text{is odd} \}.$$

\begin{defn}
\label{defn:turks-head}
For $n \in \J$, the \emph{Turk's head knots} $J_n = Th (3,n)$ are defined as the $3$-braid closures $$J_n \doteq \reallywidehat{(\sigma_1 {\sigma_2}^{-1})^n } ,$$ where a single $3$-braid $\sigma_1 {\sigma_2}^{-1}$ is depicted as follows:

\vspace{0.7 em}

\begin{center}
\begin{tikzpicture}
\pic[
rotate=0,
braid/.cd,
every strand/.style={thick},
strand 1/.style={black},
strand 2/.style={black},
strand 3/.style={black},
] {braid={s_1^{-1} s_2 }};
\end{tikzpicture}
\end{center}    
\end{defn}

\vspace{0.7 em}

The Turk's head knots $J_n$ are known to be alternating, cyclically $n$-periodic, fibered, hyperbolic, prime, strongly invertible, strongly negative amphichiral, see our recent survey paper \cite{DPS24}. 

Using Knotinfo \cite{knotinfo} and Knotscape \cite{knotscape}, we can identify the Turk's head knots $J_n$ with small number of crossings. In particular, we have $$J_2 = 4_1, \ J_4 = 8_{18}, \ J_5 = 10_{123}, \ J_7 = 14_{a19470}, \ \text{and},  J_8 = 16_{a275159} .$$

We will need the following three important properties of $J_n$. Their symmetry groups were computed by Sakuma and Weeks \cite[Proposition~I.2.5]{SW95}. The recent work of AlSukaiti and Chbili \cite[Corollary~3.5, Proposition~5.1]{AC23} provided their determinants and the roots of their Alexander polynomials. For more references, one can consult our recent survey \cite{DPS24}.

\begin{enumerate}
    \item \label{property:SW} For a fixed value of $n$, we have $$\mathrm{Sym} (S^3 , J_n) \cong D_{2n} \CommaPunct$$ where $D_{m}\cong\Z/m\Z\rtimes\Z/2\Z$ denotes the dihedral group of $2m$ elements.
    
    \item \label{eq:determinant} Let $L_k$ denote the $k^{th}$ Lucas number.\footnote{The Lucas numbers are defined recursively as $L_0 = 2, L_1 =1$, and $L_k = L_{k-1} + L_{k-2}$ for $k \geq 2$.} Then we have $$\mathrm{det}(J_n) = \Delta_{J_n} (-1) = L_{2n} - 2 \cdot$$
    
    \item \label{eq:roots} The roots of the Alexander polynomial $\Delta_{J_n} (t)$ are of the form: $$z = -\frac{1}{2} \left( 2\cos \left( \frac{2k}{n} \pi \right ) -1 \pm \sqrt{ \left (2\cos \left (\frac{2k}{n} \pi \right ) -1 \right ) ^2 -4 } \right) \CommaPunct $$ with $1 \leq k \leq \lfloor n/2 \rfloor.$
\end{enumerate}

In the following proposition, we explicitly determine the number of strong inversions for the Turk's head knots $J_n$.

\begin{prop}
\label{prop:equivalent}
The Turk's head knot $J_n$ has at most two inequivalent (i.e. not conjugate in $\mathrm{Sym}^+(S^3,J_n)$) strong inversions, say $\rho_1$ and $\rho_2$. In particular, \begin{itemize}[leftmargin=2em]
    \item if $n \in \JO$, then the strong inversion is unique and $J_n$ is Klein amphichiral,
    \item if $n\in \JE$, then $J_n$ has exactly two inequivalent strong inversions, which are conjugated by an element of $\mathrm{Sym}(S^3,J_n)$.
\end{itemize}
\end{prop}
\begin{proof}
According to \cite[Proposition 3.4]{Sak86} $J_n$ admits at most two inequivalent strong inversion, since it is invertible, amphichiral and hyperbolic. More precisely, Sakuma proved that the followings are equivalent:
\begin{itemize}[leftmargin=2em]
    \item $J_n$ admits a period $2$ symmetry,
    \item $\rho_1$ and $\rho_2$ are not equivalent.
\end{itemize}

Since by property \ref{property:SW} we have $\mathrm{Sym}(S^3,J_n)\cong D_{2n}$ and $n$ is odd, we know that we only have three conjugacy classes of involutions in $\mathrm{Sym}(S^3,J_n)$. By Proposition \ref{prob:+amphichiral} these classes are exactly given by a strongly inversion $\rho$ and a strongly negative and positive amphichiral involutions $\tau$ and $\delta$. In particular, $J_n$ cannot admit a period $2$ symmetry (cf. \cite[p.~332]{KS92}). In Figure~\ref{fig:klein_Jn} we can explicitly see that the maps $\tau$ and $\rho$ commute. Therefore, $\rho$ is the unique strong inversion.

On the other hand, if $n$ is even, then $J_n$ admits an obvious a period $2$ symmetry, which can be seen, for example from the braid description in Definition \ref{defn:turks-head}. Therefore, $\rho_1$ and $\rho_2$ are inequivalent strong inversions. 
\end{proof}

\begin{figure}[ht]
    \centering
\begin{tikzpicture}
\node[anchor=south west,inner sep=0] at (0,0){\includegraphics[scale=1]{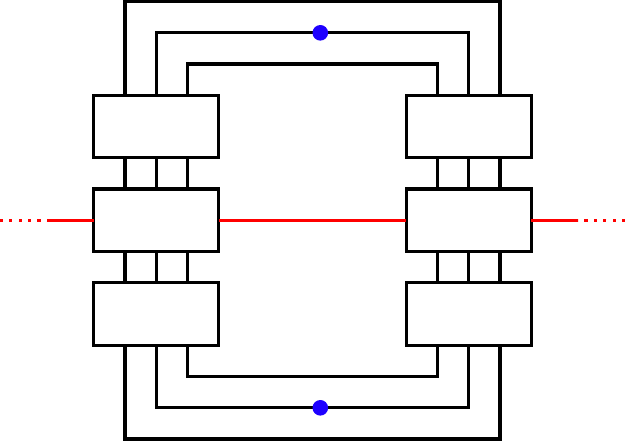}};
\node at (2.6,5.4) (b1){$(\sigma_1\sigma_2^{-1})^k$};

\node at (2.6,2.2) (b2){$(\sigma_2^{-1}\sigma_1)^k$};

\node at (8,5.4) (c1){$(\sigma_1\sigma_2^{-1})^k$};
\node at (8,2.2) (c2){$(\sigma_2^{-1}\sigma_1)^k$};

\node at (2.6,3.8) (a1) {$\alpha$};

\node at (8,3.8) (a1) {$\beta$};

\end{tikzpicture}
    \caption{The Turk's head knot $J_n$ for $n$ odd. If $n=4k+1$ then $\alpha=\sigma_1$ and $\beta=\sigma_2^{-1}$, while for $n=4k-1$ we have $\alpha=\sigma_2$ and $\beta=\sigma_1^{-1}$. Its Klein amphichiral symmetry is represented in Figure \ref{fig:10_123_sym}.}
    \label{fig:klein_Jn}
\end{figure}
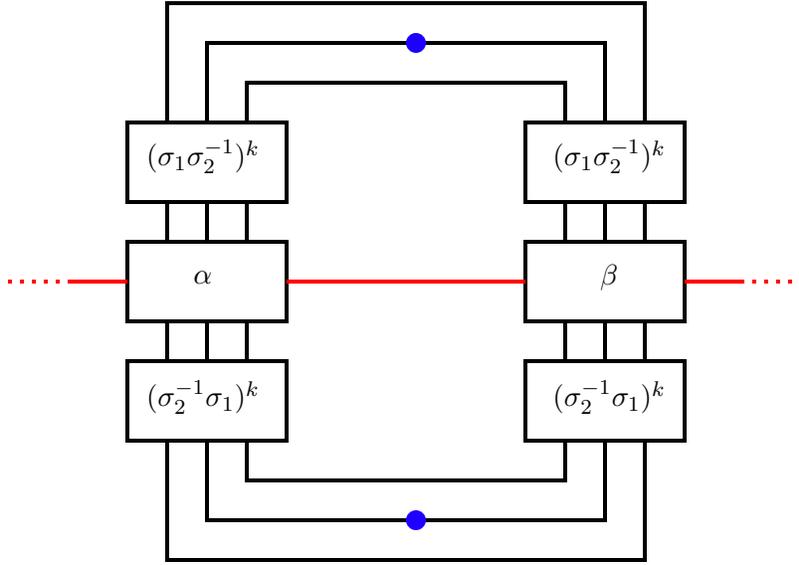

\begin{rem}
We fix here the choice of direction on $J_n$ for $n$ odd. We will always consider $J_n$ as a DSI knot with the involution $\rho$ given by the $\pi$-rotation around the red axis in Figure \ref{fig:klein_Jn} and the chosen half-axis $h$ is the bounded one in the figure, oriented from left to right.
In the following, we will omit to recall these choices. Observe that the strongly negative amphichiral involution $\tau$ (see \ref{fig:klein_Jn}) maps the DSI knot $(J_n,\rho,h)$ to $(\overline{J_n}, \rho, -h')$, where $-h'$ is the complementary half-axis endowed with the opposite orientation. Therefore, while in general these choices would be relevant for the computations in Section \ref{sec:independence}, in this specific case, changing the DSI structure would change the invariants computed in Section \ref{sec:independence} at most by a sign.
\end{rem}
\begin{cor}
If $n \in \JO$, then the Turk's head knots $J_n$ are equivariant $\Q$-slice.
\end{cor}

\begin{proof}

We know that $J_n$ is always strongly negative amphichiral. Since $n \in \JO$, by Proposition~\ref{prop:equivalent}, the knots $(J_n , \rho, \tau)$ are all Klein amphichiral. Then from Theorem~\ref{thm:mainthm1} we know that $(J_n,\rho)$ is equivariantly $\Q$-slice. 
\end{proof}

\begin{rem}
    Let $\mathrm{gcd} (p,q) = 1$. If $p$ and $q$ are both odd integers, then the Turk's head knots $Th(p,q)$ are strongly invertible and strongly negative amphichiral, see \cite[{\sc\S}3.4]{DPS24}. A positive answer to \cite[Conjecture~B]{DPS24} would also prove that the knots $Th(p,q)$ are all equivariant $\Q$-slice.
\end{rem}

\subsection{An Obstruction for Equivariant Q-Slice Knots}
\label{sec:obstructions}

In this subsection, we prove an equivariant rational version of the classical Fox-Milnor condition, which is a generalization of the result by Cochran, Franklin, Hedden, and Horn \cite{CFHH13}. Here, we normalize the Alexander polynomial of a knot $K$ such that $$\Delta_K (1) = 1 \quad \text{and} \quad \Delta_K (t) = \Delta_K (t^{-1}).$$

Now, we prove Theorem~\ref{thm:mainthm2}, claiming that the Alexander polynomial of an equivariant $\Q$-slice knot must be square.

\begin{proof}[Proof of Theorem~\ref{thm:mainthm2}]

Let $(K,\rho)$ be an equivariant $\Q$-slice knot with a slice disk $D\subset Z$ in a $\Q$-homology $4$-ball $Z$ and let $\rho$ be an involution on $Z$ extending $\rho_K$ such that $\rho(D)=D$. Let $Z_D = Z \setminus N(D)$ be the equivariant slice disk exterior, where $N(D)$ is a $\rho$-invariant tubular neighborhood of $D$. Denote the inclusion of the boundary by $i:\partial Z_D\cong S^3_0(K)\to Z_D$. 

Since $H_1(Z_D ;\Q) = H_1(S^1 \times B^3 ;\Q)$, we have $H_1(Z_D ;\Z)/ \mathrm{torsion} \cong\Z$.
Let $$\phi:\pi_1(Z_D)\to \Z$$ be the projection map on the homology modulo torsion.
Given a space $X$ and a map $\epsilon:\pi_1(X)\to\Z=\langle t\rangle$ we denote by $H_*(X,\epsilon)$ the integral homology of $X$ twisted by $\epsilon$, which is naturally a $\Z[t^{\pm1}]$-module.
By \cite[Proposition~4.6]{CFHH13}, the order of $H_1(S^3_0(K),\phi\circ i)$ is $\Delta_K(t^n)$ where $n$ is the complexity of the slice disk, i.e., the absolute value of the element represented by a meridian of $D$ in $H_1(Z_D ;\Z)/ \mathrm{torsion} \cong\Z$. Let $M$ be the kernel of
$$ i_*:H_1(S^3_0(K),\phi\circ i)\to H_1(Z_D,\phi), $$
and denote its order by $f(t)$. Then, by using \cite[Proposition~4.5]{CFHH13}, we see that $\Delta_K(t^n)=f(t)f(t^{-1})$. Now, observe that $\phi\circ\rho_*=(-id)\circ\phi$ as follows:

\begin{center}
\begin{tikzcd}
\pi_1(Z_D) \arrow{r}{\phi} \arrow[swap]{d}{\rho_*} & \Z \arrow{d}{-id } \\
\pi_1(Z_D) \arrow{r}{\phi} & \Z
\end{tikzcd}
\end{center}

Therefore the order of $\rho(M)$ is $f(t^{-1})$. Since the inclusion map $i$ is $\rho$-equivariant we have that $\rho(M)=M$, hence $f(t)=f(t^{-1})$. Hence $\Delta_K(t^n)$ is a square. In turn, this implies that $\Delta_K(t)$ is also a square.
\end{proof}

Using the equivariant Fox-Milnor condition in Theorem~\ref{thm:mainthm2}, we now prove that the other half of the knots $J_n$ are not trivial in $\CTQ$.

\begin{prop}
If $n \in \JE$, then the Turk's head knots $J_n$ are not equivariant $\Q$-slice.    
\end{prop}

\begin{proof}
It is a well-known fact that Lucas numbers satisfy the following equality:
$$
(L_n)^2=L_{2n}+(-1)^n\cdot2.
$$
If $n\in \JE$, then $\det(J_n)=L_{2n}-2=(L_n)^2-2$ by the identity (\ref{eq:determinant}). Since the determinant is not a perfect square, for an arbitrary $n$, $J_n$ is not equivariant $\Q$-slice by Theorem \ref{thm:mainthm2}.
\end{proof}

Generalizing Kirby calculus argument by Fintushel and Stern \cite{FS84}, in \cite[Theorem~4.14]{Cha07} Cha exhibits a familiy of infinitely many $\Q$-slice knots $K_n$ in $S^3$, depicted in Figure~\ref{fig:cha_knots}. Since the knots $K_n$ are clearly strongly negative amphichiral (see \cite[Figure~5]{Cha07}), Cha's result can be reproven by Kawauchi's characterization. Note that $K_1$ is the figure-eight knot.

\begin{figure}[ht]
    \centering
\begin{tikzpicture}
\node[anchor=south west,inner sep=0] at (0,0){\includegraphics[scale=0.6]{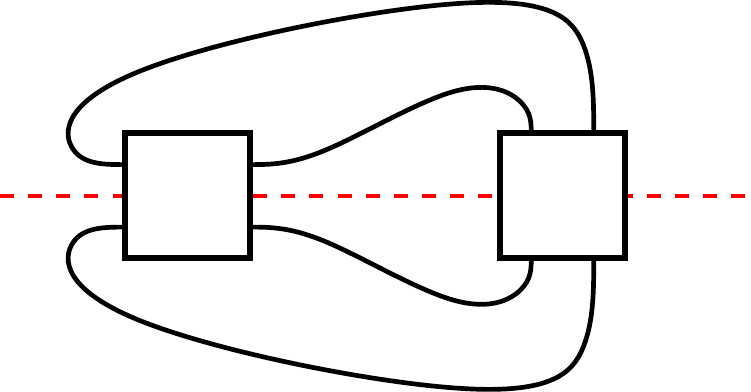}};
\node[scale=1.7] at (1.9,2) (b4){$-n$};
\node[scale=1.7] at (5.8,1.99) (b37){$n$};

\end{tikzpicture}
    \caption{The knots $K_n$. The square box with the integer $n$ (resp. $-n$) represents the right-handed (resp. the left-handed) $n$ full twists. The strong inversion is the $\pi$-rotation around the dashed axis.}
    \label{fig:cha_knots}
\end{figure}

Now, we are ready to prove Theorem \ref{thm: mainthm4}, showing that Cha's knots $K_n$ are also not equivariant $\Q$-slice.

\begin{proof}[Proof of Theorem \ref{thm: mainthm4}]

Assume that $K_n$ is equivariant $\Q$-slice for each $n \geq 1$. By \cite[Theorem~4.14]{Cha07}), we know that $$\Delta_{K_n} (t) = -n^2 t^{-1} + 2n^2 + 1 - n^2 t .$$ Then we get $$\mathrm{det} (K_n) = \Delta_{K_n} (-1) = 4n^2 + 1 = (2n^2)^2 + 1.$$ Then, by Theorem \ref{thm:mainthm2}, the determinant must be a square, which is a contradiction. Therefore, the knots $K_n$ are not equivariant $\Q$-slice.

Let $H$ be the subgroup of $\widetilde{\C}_\Q$ spanned by $K_n$ for $n\geq1$ (fix any choice of strong inversion and direction on $K_n$). Let $K=\widetilde{\#}^{a_1}K_{n_1}\widetilde{\#}\dots\widetilde{\#}^{a_l}K_{n_l}, $ where the equivariant connected sum is taken with respect to any ordering of the knots. 

Since the polynomials $\Delta_{K_n}(t)$ are quadratic, it is not difficult to check that they are pairwise coprime. Therefore, $\Delta_K(t)$ is a square if and only if $a_i=0\mod{2}$ for $i=1,\dots,l$. Hence, the group $H$ surjects onto $(\Z/2\Z)^\infty$.
\end{proof}

\begin{rem}
\label{rem:algebraic concordance}
One can easily check by using the so-called \emph{equivariant signature jumps homomorphism} $\widetilde{J}_\lambda:\widetilde{\C}\to\Z$ introduced in \cite[Definition 6.1]{DP23b} and by applying \cite[Theorem 6.9]{DP23b} that the knots $K_n$ span a subgroup of $\widetilde{\C}$ which surjects onto $\Z^\infty$. As the classical Levine-Tristram signatures provide invariants of $\Q$-concordance (see \cite{CK02}), we expect the maps $\widetilde{J}_\lambda$ to factor through $\widetilde{\C}_\Q$. This would prove that the subgroup $H\subset\widetilde{\C}_\Q$ described in the proof of Theorem~\ref{thm: mainthm4} would actually surject onto $\Z^\infty$. Compare with Problem~\ref{prob:algebraic concordance}.
\end{rem}

\section{Proof of Theorem \ref{thm: mainthm3}}
\label{sec:independence}
In this section, we recall some preliminary notions and we prove the intermediate results needed in order to prove Theorem \ref{thm: mainthm3}. 

\subsection{Weighted Graphs, Spanning Trees and Gordon-Litherland form}
In this subsection, we recall some useful results about the Gordon-Litherland form and graph theory.
\label{sec:graphs}
\begin{defn}
A \emph{weighted graph} $\Gamma=(V,E)$ is a simple graph with vertex set $V$ and edges $E$ with the additional data of a \emph{weight} $\lambda_{e}=\lambda_{ij}=\lambda_{ji}\in\R$, if there exists $e\in E$ connecting $i,j\in V$. If two vertices are not connected by an edge the weight is understood to be zero.
We associate a \emph{Laplacian matrix} $\L(\Gamma)$ for each weighted graph by letting $$
\L(\Gamma)_{i,j}=\begin{cases}
    -\lambda_{ij}\quad\text{if}\quad i\neq j\\
    \sum_{k\neq i}\lambda_{ik}\quad\text{if}\quad i=j.
\end{cases}
$$
In the following, we will denote by $\L(\Gamma;i)$ the square matrix obtained from $\L(\Gamma)$ by removing the $i$-th column and row.
\end{defn}

\begin{defn}
Let $\Gamma=(V,E)$ be a weighted graph. A \emph{spanning tree} of $G$ is subgraph $T=(V_T,E_T)$ of $\Gamma$ such that
\begin{itemize}
    \item $T$ is a tree,
    \item the vertex set of $T$ is equal to $V$.
\end{itemize}
We denote the \emph{weight} of $T$ by
$$
w(T)=\prod_{e\in E_T}\lambda_e.
$$
Finally, we define the \emph{(weighted) number of spanning trees} of $\Gamma$ as
$$
\T(\Gamma)=\sum_{T\subset\Gamma\;\text{spanning tree}}w(T).
$$
\end{defn}

\begin{thm}\cite[Theorem VI.29]{Tut01}\label{thm:matrix_tree}
Let $\Gamma=(V,E)$ be a weighted graph. Then for every $i\in V$, we have
$$
\det(\L(\Gamma;i))=\T(\Gamma).
$$
\end{thm}

\begin{thm}\cite[Theorem 6.1.1 (Gershgorin's Theorem)]{HJ85}\label{thm:gershgorin}
Let $A=(a_{i,j})$ be a $n\times n$ square complex matrix, and denote by $R_i(A)=\sum_{j\neq i}|a_{i,j}|$.
Then the eigenvalues of $A$ are contained in the following set
$$
\bigcup_{i}\{z\in\mathbb{C}\;|\;|z-a_{i,i}|\leq R_i(A)\}.
$$
\end{thm}

\begin{defn}\label{def:dom_diag}
Let $A=(a_{i,j})$ be as above. We say that $A$ is \emph{dominant diagonal} if for every $1\leq i\leq n$ we have
$$
|a_{i,i}|\geq R_i(A).
$$
Moreover, if there exist $i$ such that $|a_{i,i}|>R_i(A)$ then we say that $A$ is \emph{strongly dominant diagonal}.
\end{defn}

\begin{cor}\label{cor:pos_def}
If $A=(a_{i,j})$ is a symmetric matrix and strongly dominant diagonal matrix with positive entries on the diagonal then $A$ is positive definite.
\end{cor}

\begin{exmp}
Let $\Gamma=(V,E)$ be a connected and weighted graph such that $\lambda_e>0$ for every $e\in E$. Then for every $i\in V$ the matrix $\L(\Gamma;i)$ is strongly dominant diagonal and all the diagonal entries are positive, hence $\L(\Gamma;i)$ is positive definite.
\end{exmp}

\begin{defn}
Let $\Gamma=(V,E)$ be a weighted graph. Given an edge $e\in E$ we denote by $\Gamma\setminus e=(V,E\setminus{e})$ the weighted graph obtained from $\Gamma$ by \emph{edge deletion} of $e$.
Let $i,j$ be the endpoints of $e$. We define $\Gamma/e=(\overline{V},\overline{E})$ as the graph obtained from $\Gamma$ by \emph{edge contraction} of $e$, where
\begin{itemize}
    \item $\overline{V}=V/{i\sim j}$,
    \item if $e\in E$ has endpoints $h,k\in V\setminus\{i,j\}$ then $e\in\overline{E}$ with the same label,
    \item for every $h\in V\setminus\{i,j\}$ there is an edge $e\in\overline{E}$ joining $[i\sim j]$ to $h$ of weight $\lambda_e=\lambda_{ih}+\lambda_{jh}$.
\end{itemize}
\end{defn}

\begin{thm}[\cite{tutte2004graph}]\label{thm:DCT}
Let $\Gamma=(V,E)$ be a weighted graph. Then for every $e\in E$, we have
$$
\T(\Gamma)=\T(\Gamma\setminus e)+\lambda_e\cdot\T(\Gamma/e).
$$
\end{thm}

\subsubsection{Gordon-Litherland Form}\label{sec:GL_form}

Let $L\subset S^3$ be a link and let $D\subset S^2$ be a connected diagram for $L$. Recall that by coloring $S^2\setminus D$ in a checkerboard fashion, we determine two spanning surfaces for $L$. Denote by $F$ the surface given by the black coloring. 

Then, we can describe the \emph{Gordon-Litherland form} conveniently in terms of a weighted graph $\Gamma=(V,E)$ associated with $F$, as follows. See \cite{GL78} for more details. Let $V$ be the set of white regions of $D$. Given two white regions $R_1,R_2$, we have an edge $e$ connecting them if they share at least one crossing of $D$. The label of $e$ is given by the sum over all the crossing of $D$ shared by $R_1$ and $R_2$ of $+1$ for a right-handed half-twist and $-1$ for a left-handed half-twist (see Figure \ref{fig:goeritz_crossings}).

\begin{figure}[ht]
    \centering
\begin{tikzpicture}
\node[anchor=south west,inner sep=0] at (0,0){\includegraphics[scale=0.7]{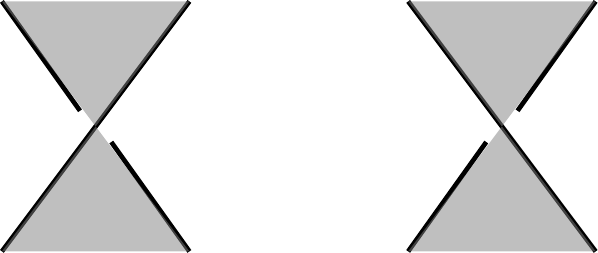}};
\node at (1.2,-0.5) (b1){$+1$};
\node at (0.5,1.5) (b3){$R_1$};
\node at (1.7,1.5) (b4){$R_2$};
\node at (5.3,1.5) (b37){$R_1$};
\node at (6.5,1.5) (b43){$R_2$};
\node at (6,-0.5) (b2){$-1$};

\end{tikzpicture}
    \caption{Sign convention for the crossings.}
    \label{fig:goeritz_crossings}
\end{figure}
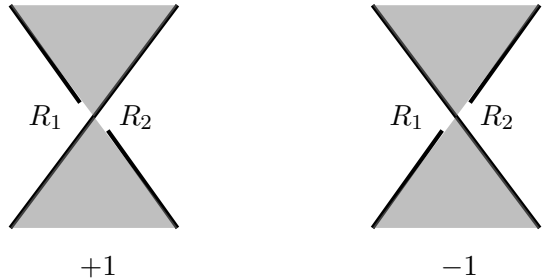
Then, for any $R\in V$ the matrix $\L(\Gamma;R)$ represents the Gordon-Litherland form of $F$.
In particular, the determinant of the link $L$ is given by the absolute value of $\det(\L(\Gamma;R))=\T(\Gamma)$.

\subsection{Butterfly and Moth Links}\label{ssec:buttefly_moth}
We are now going to recall the definition of \emph{$n$-butterfly link} that is defined by the first author in \cite{DP23} as a generalization of the \emph{butterfly link} introduced by Boyle and Issa in \cite{BI22}.

\begin{defn}\label{butterfly_link}

Let $(K,\rho,h)$ be a DSI knot. Take a $\rho$-invariant band $B$, containing the half-axis $h$, which attaches to $K$ at the two fixed points. Performing a band move on $K$ along $B$ produces a $2$-component link, which has a natural semi-orientation induced by the unique semi-orientation on $K$. The linking number between the components of the link depends on the number of twists of the band $B$. Then, the \emph{$n$-butterfly link} $L_b^n(K)$, is the $2$-component $2$-periodic link (i.e. the involution $\rho$ exchanges its components) obtained from such a band move on $K$, so that the linking number between its components is $n$.
\end{defn}

\begin{rem}
In order to avoid confusion, we want to remark that the definition of $L_b^n(K)$ given in Definition \ref{butterfly_link} above actually coincide with the definition of $\widehat{L}_b^{-n}(K)$ given in \cite[Definition 1.9]{DP23}. We choose to use a different notation to improve readability.
\end{rem}

In \cite{DPF23b}, Framba and the first author associate with a DSI knot the so called \emph{moth link} which we are now going to recall (see \cite[Definition 5.2]{DPF23b} for details).

\begin{defn}\label{def:moth_link}
Let $(K,\rho,h)$ be a DSI knot and let $B$ be the invariant band giving the $0$-butterfly link $L_b^0(K)$, as in Definition \ref{butterfly_link}. Observe that we can undo the band move on $B$ by attaching another invariant band $B^*$ on $L_b^0(K)$. Then we define the \emph{moth link} $L_m(K)$ as the \emph{strong fusion} (see \cite{Kai92}) of $L_b^0(K)$ along the band $B^*$.
\end{defn}

Now, the \emph{moth polynomial} of $(K,\rho,h)$ is defined as the Kojima-Yamasaki eta-function (see \cite{KY79}) of the moth link of $K$.

\begin{prop}\cite[Proposition 5.6]{DPF23b}
The moth polynomial induces a group homomorphism
\begin{align*}
 \eta_m : \ & \CT \to \Q(t) \\
 & K \mapsto \eta(L_m(K))(t).
\end{align*}
\end{prop}

\begin{prop}\cite[Proposition 5.7]{DPF23b}\label{prop:eta_conway}
    The moth polynomial of a DSI knot $K$ can be computed by the following formula:
    $$
\eta_m(K)(t)=\frac{\nabla_{L_b^0(K)}(z)}{z\nabla_K(z)},
$$
    where $\nabla_L(z)$ is the Conway polynomial of a (semi)-oriented link $L$ and $z=i(2-t-t^{-1})^{1/2}$.
\end{prop}

\begin{rem}
In the following, we will regard the moth polynomial as a group homomorphism
$$
\eta_m: \CT \to \Q(z)
$$
defined by the formula in Proposition \ref{prop:eta_conway}.
\end{rem}

\begin{lem}\label{lemma:skein_conway}
Let $K$ be a DSI knot. Then for any $p,q\in\Z$
$$
\nabla_{K}(z)|\nabla_{L_b^p(K)}(z)\iff\nabla_{K}(z)|\nabla_{L_b^q(K)}(z).$$
\end{lem}
\begin{proof}
Let $p\in\Z$ be any integer. It is sufficient to prove that $$\nabla_{K}(z)|\nabla_{L_b^p(K)}(z)\iff\nabla_{K}(z)|\nabla_{L_b^{p+1}(K)}(z) .$$
In order to do so we can apply the skein relation for the Conway polynomial as indicated in Figure \ref{fig:skein_butterfly} to obtain
$$
\nabla_{L_b^{p+1}(K)}(z)=\nabla_{L_b^p(K)}(z)+z\nabla_K(z).
$$
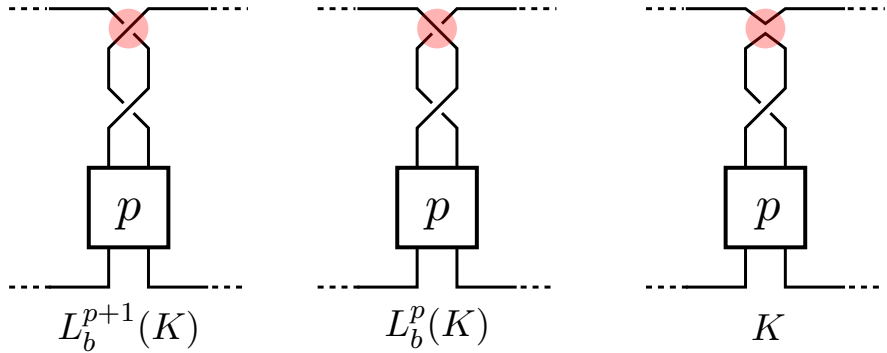
\begin{figure}[ht]
    \centering
\begin{tikzpicture}
\node[anchor=south west,inner sep=0] at (0,0){\includegraphics[scale=0.5]{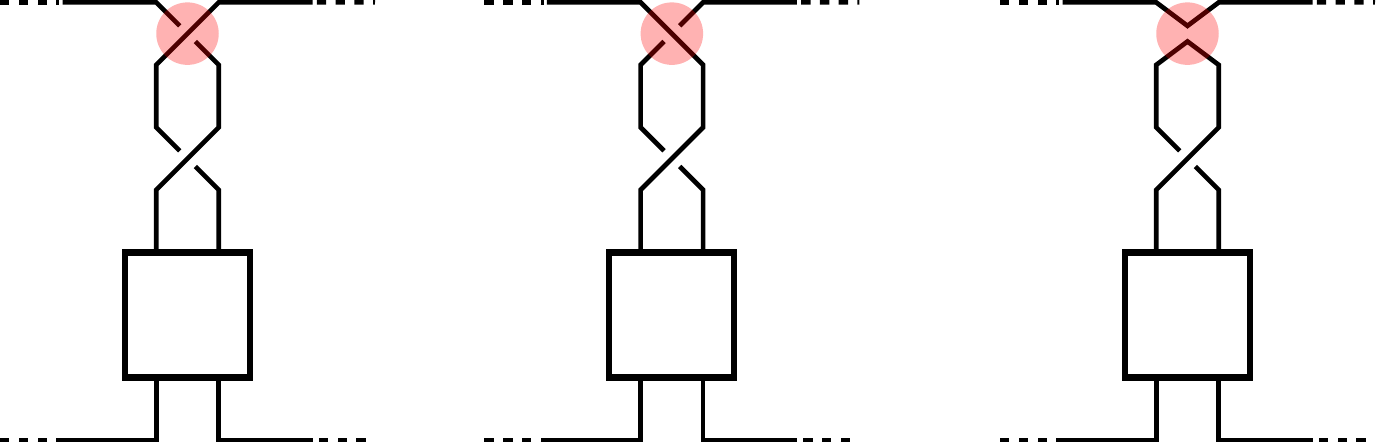}};
\node[scale=1.7] at (1.6,1) (b4){$p$};
\node[scale=1.7] at (5.7,1) (b37){$p$};
\node[scale=1.7] at (10.1,1) (b7){$p$};
\node[scale=1.3] at (1.6,-0.5) (a){$L^{p+1}_b(K)$};
\node[scale=1.3] at (5.7,-0.5) (b){$L^{p}_b(K)$};
\node[scale=1.3] at (10.1,-0.5) (c){$K$};

\end{tikzpicture}
    \caption{The three terms appearing in the skein relation. The box denotes $p$ full twists.}
    \label{fig:skein_butterfly}
\end{figure}
\end{proof}
\begin{cor}\label{cor:eta_determinant}
Let $K$ be a DSI knot and let $\eta_m(K)(z)=f(z)/g(z)$, where $f(z),g(z)\in\Z[z]$ are coprime polynomials. Suppose that for some $p\in\Z$ we have that $\det(K)$ does not divide $\det(L_b^p(K))$. Then $\deg g(z)>0$ and $g(z)|\nabla_K(z)$.
\end{cor}
\begin{proof}
Recall that for a link $L$ we have that $\det(L)=|\nabla_L(-2i)|$. Since $\det(K)$ does not divide $\det(L_b^p(K))$, it follows that $\nabla_K(z)$ does not divide $\nabla_{L_b^{p}(K)}(z)$. Hence, by Lemma \ref{lemma:skein_conway}, $\nabla_K(z)$ does not divide $\nabla_{L_b^{0}(K)}(z)$ either. The proof follows by observing that $z|\nabla_{L_b^0(K)}(z)$, since $L_b^0(K)$ is a link with more than one component.
\end{proof}

We are now going to use the moth polynomial to prove the main result of this section.

\begin{rem}
In \cite{Sak86} Sakuma introduced an equivariant concordance invariant $\eta_{(K,\rho)}$ obtained by taking the Koijima-Yamasaki eta-function of a certain link associated with $(K,\rho)$. However, thanks to the symmetries of $J_n$ and \cite[Proposition 3.4]{Sak86}, we know that Sakuma's eta-polynomial always vanishes for $J_n$, $n\in\JO$.
\end{rem}

\begin{lem}\label{lem:det_int}
Let $n \in \JO $. Then there exists $p\in\Z$ such that
$$
\frac{\det(L_b^p(J_n))}{\det(J_n)}\text{ is not an integer}.
$$
\end{lem}
\begin{proof}
Let $F_n$ be the spanning surface for $J_n$ depicted in Figure \ref{fig:spanning_Jn}, where $n=2k+1$.
\vspace{1em}

\begin{figure}[ht]
    \centering
\begin{tikzpicture}
\node[anchor=south west,inner sep=0] at (0,0){\includegraphics[scale=0.6]{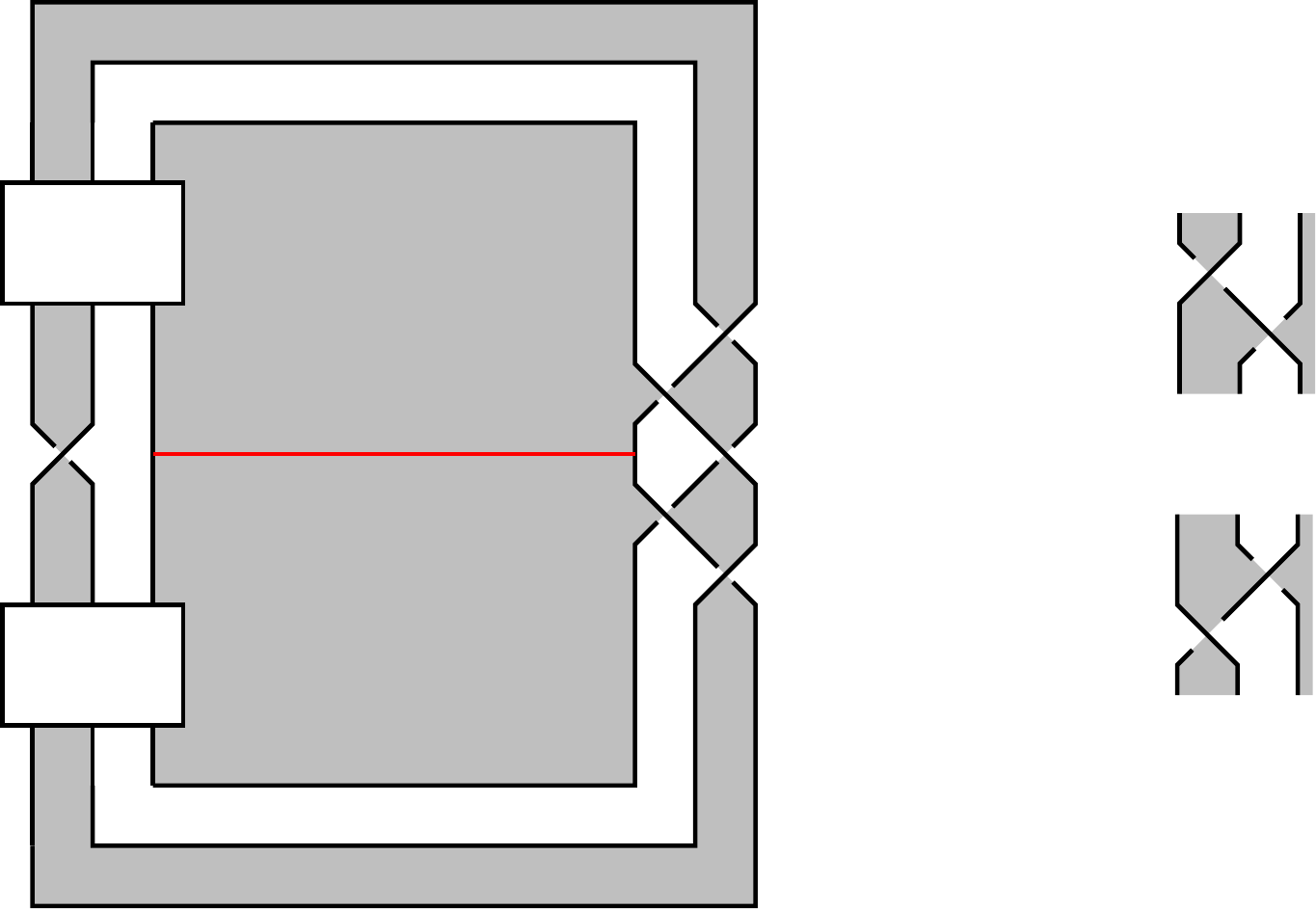}};
\node at (1,7) (b1){$(\sigma_1\sigma_2^{-1})^k$};
\node at (1,2.6) (b2){$(\sigma_2^{-1}\sigma_1)^k$};
\node at (11,6.3) (c1){$\sigma_1\sigma_2^{-1}=$};
\node at (11,3.2) (c2){$\sigma_2^{-1}\sigma_1=$};
\node at (4,5) (h){$h$};

\end{tikzpicture}
    \caption{The spanning surface $F_n$ for $J_n$}
    \label{fig:spanning_Jn}
\end{figure}
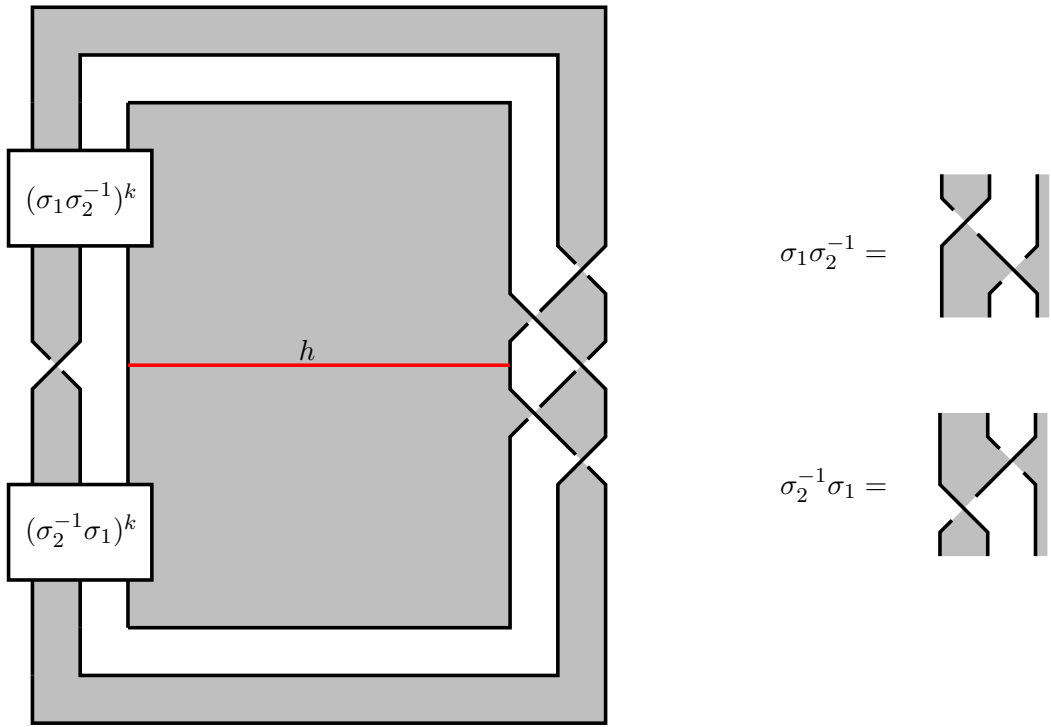

Observe that by cutting $F_n$ along the half-axis $h$, we get a spanning surface $\overline{F_n}$ for the $p$-butterfly link of $J_n$ for some $p\in\Z$.
We are now going to show that
\begin{equation}\label{eq:det_ineq}
2\det(J_n)<\det(L^p_b(J_n))<4\det(J_n),
\end{equation}
which is sufficient, since $\det(J_n)$ is odd, while $\det(L_b^p(J_n))$ is even, hence it cannot be that $$\det(L_b^p(J_n))=3\det(J_n).$$ 

Let $\Gamma_n$ (see Figure \ref{fig:goeritz_graph}) be the graph associated with $F_n$ (see Section \ref{sec:GL_form}).
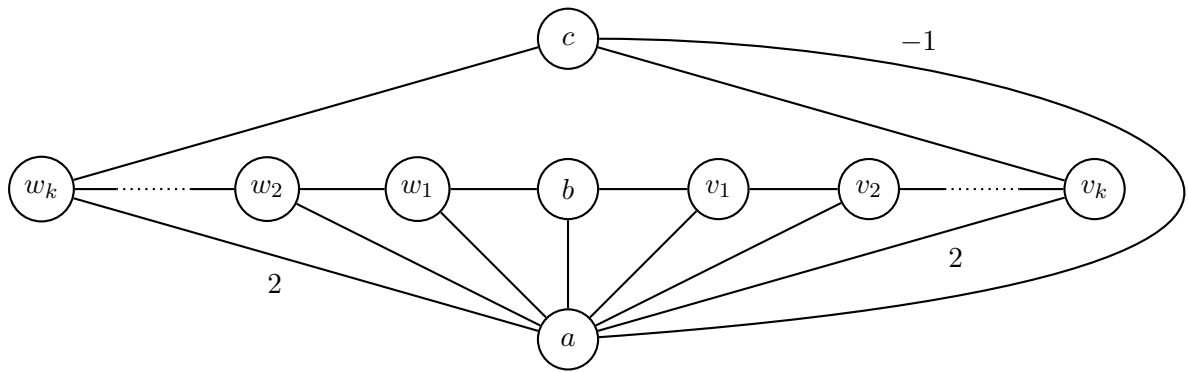
\begin{figure}[ht]
    \centering

\begin{tikzpicture}[thick,main/.style = {draw, circle},minimum width =8mm]

\node[main] at (0,0) (b){$b$};
\node[main] at (0,-2) (a){$a$};
\node[main] at (0,2) (c){$c$};
\node[main] at (2,0) (v1){$v_1$};
\node[main] at (4,0) (v2){$v_2$};
\node[main] at (7,0) (vk){$v_k$};
\node[main] at (-2,0) (w1){$w_1$};
\node[main] at (-4,0) (w2){$w_2$};
\node[main] at (-7,0) (wk){$w_k$};

\draw[-] (a) -- (b);
\draw[-] (b) -- (v1);
\draw[-] (v1) -- (v2);
\draw[-] (v2) -- (5,0);
\draw[dotted] (5,0) -- (6,0);
\draw (6,0) -- (vk);
\draw[-] (b) -- (w1);
\draw[-] (w1) -- (w2);
\draw[-] (a) -- (v1);
\draw[-] (a) -- (v2);
\draw[-] (a) -- (w1);
\draw[-] (a) -- (w2);
\draw[-] (w2) -- (-5,0);
\draw[dotted] (-5,0) -- (-6,0);
\draw (-6,0) -- (wk);

\draw (vk) -- (c) -- (wk);

\draw[-] (a) -- node[midway, below left] {$2$} (wk) ;
\draw[-] (a) -- node[pos=0.7, below right] {$2$} (vk) ;

\draw (a) .. controls +(14,1) and +(7,0) .. (c) node[pos=0.8, above right] {$-1$};
\end{tikzpicture}

    \caption{The graph associated with $F_n$.}
    \label{fig:goeritz_graph}
\end{figure}
The corresponding graph $\overline{\Gamma_n}$ for $\overline{F_n}$ is easily obtained from $\Gamma_n$ by identifying the vertices $b$ and $c$.
Recall that from Theorem \ref{thm:matrix_tree} and the discussion in Section \ref{sec:GL_form}, we have
\begin{equation}
\label{eq:det_eq}
\det(J_n)=\T(\Gamma_n)\quad\text{and}\quad\det(L_p(J_n))=\T(\overline{\Gamma_n}).
\end{equation}

Denote by $\Gamma_n(x)$ be the graph obtained by adding an edge $e$ with label $x$ between $b$ and $c$ to $\Gamma_n$.
Applying Theorem \ref{thm:DCT} to the edge $e$ of $\Gamma_n(x)$, we get
\begin{equation}\label{eq:gamma_nx}
\T(\Gamma_n(x))=\T(\Gamma_n)+x\T(\overline{\Gamma_n}).
\end{equation}

Using (\ref{eq:det_eq}) and (\ref{eq:gamma_nx}), we can see that the inequalities (\ref{eq:det_ineq}) are equivalent to showing that
$$
\T(\Gamma_n(-1/4))>0\quad\text{and}\quad\T(\Gamma_n(-1/2))<0.
$$
Observe now that the matrix $\L(\Gamma_n(-1/4);a)$ (see Section \ref{sec:graphs}) is positive definite: multiplying by $3$ both the row and column corresponding to the vertex $c$ we get a dominant diagonal matrix with positive diagonal entries, which has all positive eigenvalues by Gershgorin Theorem. Hence $\T(\Gamma_n(-1/4))=\det(\L(\Gamma_n(-1/4);a))>0$.

On the other hand, it is not difficult to see that $\L(\Gamma_n(-1/2);a)$ has an inertia $(n,1,0)$ i.e. it is nonsingular and it has $n$ positive eigenvalues and $1$ negative eigenvalue. Removing the column and row corresponding to $c$ we get a positive definite matrix by Gershgorin Theorem, hence $\L(\Gamma_n(-1/2);a)$ has at least $n$ positive eigenvalues. However, $\L(\Gamma_n(-1/2);a)$ has at least (and hence exactly) one negative eigenvalue: the restriction to the subspace spanned by the vertices $c,b,v_k,w_k$ is given by
$$
\begin{pmatrix}
    1/2&1/2&-1&-1\\
    1/2&5/2&0&0\\
    -1&0&4&0\\
    -1&0&0&4
\end{pmatrix}
$$
which has negative determinant. In particular, $\T(\Gamma_n(-1/2))=\det(\L(\Gamma_n(-1/2);a))<0$.

\end{proof}

\begin{prop}\label{prop:lin_indep}
Let $\mathcal{F} \subset \JO$ be an infinite family such that if $m,n\in\mathcal{F}$ and $n\neq n$ then $m$ and $n$ are coprime. Let $J_\mathcal{F}$ be the subgroup of $\CT$ generated by $\{J_n\;|\;n\in\mathcal{F}\}$. Then
$$ J^{ab}_\mathcal{F} = J_\mathcal{F}/\left[J_\mathcal{F},J_\mathcal{F}\right]\cong \Z^\infty. $$
\end{prop}

\begin{proof}
We prove that $\{\eta_m(J_n)(z)\;|\;n\in\mathcal{F}\}$ are $\Z$-linearly independent in $\Q(z)$. If $\operatorname{gcd} (p,q)=1$, by the property (\ref{eq:roots}), the sets of roots of the Alexander polynomials of $J_p$ and $J_q$ do not intersect, hence $\Delta_{J_p}(t)$ and $\Delta_{J_q}(t)$ are coprime polynomials. This in turns implies that the Conway polynomials $\nabla_{J_p}(t)$ and $\nabla_{J_q}(t)$ are coprime. Then the linear independence follows by the use of Lemma \ref{lem:det_int} and Corollary \ref{cor:eta_determinant}.
\end{proof}

\subsection{String links and Milnor invariants}
In \cite{DPF23} Framba and the first author used the close relation between strongly invertible knots and \emph{string links} to prove that the equivariant concordance group $\CT$ is not solvable.
In particular, they considered a homomorphism
$$
\widetilde{\C}\xrightarrow{\varphi\circ\pi}\C(2),
$$
where $\C(2)$ is the \emph{concordance group of string links on 2 strings}, introduced in \cite{LeD88} (see \cite[Section 3]{DPF23} for details), and they proved that $\widetilde{\C}$ is not solvable by using Milnor invariants for string links (see \cite{HL98}).

We now use a similar approach in order to prove that two given knots in $\widetilde{\C}$ do not commute: we determine their image in $\C(2)$ and then we compute the Milnor invariants of their commutator to show that it is nontrivial. For details, one can consult \cite{DPF23}.

\begin{lem}\label{lemma:commutator}
The commutator between $J_5$ and $J_7$ is not equivariantly slice.
\end{lem}
\begin{proof}
Following the procedure described in \cite[Remark 4.5]{DPF23}, we determine the image of $[J_5,J_7]=J_5\widetilde{\#}J_7\widetilde{\#}J_5^{-1}\widetilde{\#}J_7^{-1}$ in $\C(2)$, depicted in Figure \ref{fig:commutator}.

\begin{figure}[ht]
    \centering
    
\begin{tikzpicture}
\node[anchor=south west,inner sep=0] at (0,0){\includegraphics[scale=0.5]{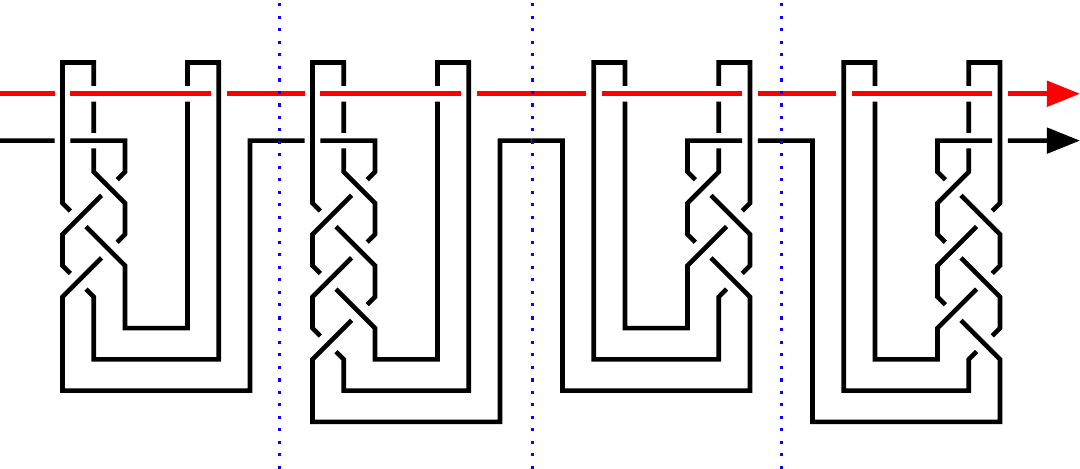}};
\node at (1.2,4) (b3){$J_5$};
\node at (3.5,4) (b4){$J_7$};
\node at (5.7,4) (b37){$J_5^{-1}$};
\node at (7.8,4) (b43){$J_7^{-1}$};
\end{tikzpicture}

\vspace{3mm}
\raggedright
\begin{itemize}
\color{red}
    \item[\texttt{l1: }]\texttt{30 26 51 43 65 61 84 76}
\color{black}
\item[\texttt{l2: }] \texttt{30 34 35 4 24 4 22 27 51 47 46 7 56 53 37 7 39 42 64 69 11 67 11 56 57 61 85 88 16 90 74 71 16 81 80 76}
\end{itemize}
    \caption{The $2$-string link representing $\varphi\circ\pi([J_5,J_7])$.}
    \label{fig:commutator}
\end{figure}

Using the computer program \texttt{stringcmp} \cite{TKS13}, we find out that $\varphi\circ\pi([J_5,J_7])$ has nontrivial Milnor invariants. The first nontrivial invariant is in degree 6. In Figure \ref{fig:commutator}, we also report the input data needed to run \texttt{stringcmp}, which codifies the longitudes of the two components of the string link.
\end{proof}
\begin{proof}[Proof of Theorem \ref{thm: mainthm3}]
Let $\mathcal{J}$ be the subgroup of $\widetilde{\C}$ generated by $\{J_p\;|\;p\geq5\text{ prime}\}$. By Proposition \ref{prop:lin_indep}, we know that $\mathcal{J}^{ab}\cong\Z^\infty$. Moreover, it is spanned by negative amphichiral knots, therefore $\mathfrak{f}(\mathcal{J})\subset\C$ is $2$-torsion. Finally, by Lemma \ref{lemma:commutator}, we know that $\mathcal{J}$ is not abelian.
\end{proof}

\bibliography{references}
\bibliographystyle{amsalpha}

\end{document}